\documentclass[12pt]{amsart}
\renewcommand*{\marginpar}[1]{} 
\usepackage[all]{xy}

\usepackage{amsmath, amsthm, amssymb,latexsym, graphics}

\newcommand{\N}{\mathbb{N}}
\newcommand{\Z}{\mathbb{Z}}

\newcommand{\R}{\mathbb{R}}
\newcommand{\C}{\mathbb{C}}
\newcommand{\E}{\mathbb{E}}

\newcommand{\spr}[2]{\langle #1, #2 \rangle}
\newcommand{\ck}{\check{\phantom{i}}}

\newcommand{\HI}{H^\infty}

\DeclareMathOperator{\loc}{loc}

\newcommand{\Sobolev}{W}

\newcommand{\Sob}{\Sobolev_2}
\newcommand{\Soa}{\Sobolev^\alpha_2}

\newcommand{\Soaloc}{\Sobolev^\alpha_{2,\loc}}
\newcommand{\Wor}{\widetilde{\mathcal{H}}}
\newcommand{\Wa}{\Wor^\alpha_2}

\newcommand{\Hor}{\mathcal{H}}
\newcommand{\Ha}{\Hor^\alpha_2}

\newcommand{\Sobexp}{\mathcal{W}}
\newcommand{\Sea}{\Sobexp^\alpha_2}

\newcommand{\equi}{\psi}
\newcommand{\tdyad}{\widetilde{\dyad}}
\newcommand{\Fdyad}{\phi}
\newcommand{\dyad}{\varphi}

\newcommand{\ta}{\langle t \rangle^\alpha}
\newcommand{\tma}{\langle t \rangle^{-\alpha}}

\DeclareMathOperator{\supp}{supp}

\DeclareMathOperator{\vect}{span}

\DeclareMathOperator{\Str}{Str}
\DeclareMathOperator{\Hol}{Hol}

\DeclareMathOperator{\sgn}{sgn}

\let\Re=\relax \DeclareMathOperator{\Re}{Re}
\let\Im=\relax \DeclareMathOperator{\Im}{Im}

\newcommand{\bignorm}[1]{\Bigl\Vert#1\Bigr\Vert}
\newcommand{\Bignorm}[1]{\Bigl\Vert#1\Bigr\Vert}

\newtheorem{thmalt}{Theorem}[section]

\theoremstyle{definition}
\newtheorem{rem}[thmalt]{Remark}
\newtheorem{defi}[thmalt]{Definition}
\newtheorem{thm}[thmalt]{Theorem}

\newtheorem{lem}[thmalt]{Lemma}
\newtheorem{prop}[thmalt]{Proposition}
\newtheorem{exa}[thmalt]{Example}

\newtheorem{nota}[thmalt]{Notation}

\numberwithin{equation}{section}

\title[Spectral multiplier theorems and averaged $R$-boundedness]
 {Spectral multiplier theorems and averaged $R$-boundedness} 

\author[Ch. Kriegler]{Christoph Kriegler}
\address{Christoph Kriegler\\
Laboratoire de Math\'ematiques (CNRS UMR 6620)\\
Universit\'e Blaise-Pascal (Clermont-Ferrand 2)\\
Campus Universitaire des C\'ezeaux\\
3, place Vasarely\\
TSA 60026\\
CS 60026\\
63 178 Aubi\`ere Cedex\\
France
}
\email{christoph.kriegler@math.univ-bpclermont.fr}
\thanks{The first named author acknowledges financial support from the Franco-German University (DFH-UFA) and the Karlsruhe House of Young Scientists (KHYS).
The second named author acknowledges the support by the DFG through CRC 1173.}

\author[L. Weis]{Lutz Weis}
\address{Lutz Weis\\
Karlsruher Institut f\"ur Technologie\\
Fakult\"at f\"ur Mathematik\\
Institut f\"ur Analysis\\
Englerstra\ss{}e 2\\
76131 Karlsruhe\\
Germany}
\email{lutz.weis@kit.edu}

\date{\today}

\allowdisplaybreaks

\begin{document}


\begin{abstract}
Let $A$ be a $0$-sectorial operator with a bounded $H^\infty(\Sigma_\sigma)$-calculus for some $\sigma \in (0,\pi),$ e.g. a Laplace type operator on $L^p(\Omega),\: 1 < p < \infty,$
where $\Omega$ is a manifold or a graph.
We show that $A$ has a $\Ha(\R_+)$ H\"ormander functional calculus if and only if certain operator families derived from the resolvent $(\lambda - A)^{-1},$ the semigroup $e^{-zA},$ the wave operators $e^{itA}$ or the imaginary powers $A^{it}$ of $A$ are $R$-bounded in an $L^2$-averaged sense.
If $X$ is an $L^p(\Omega)$ space with $1 \leq p < \infty,$ $R$-boundedness reduces to well-known estimates of square sums.
\end{abstract}

\maketitle

\section{Introduction}\label{Sec 1 Intro}

H\"ormander's Fourier multiplier theorem states that for a function $f \in \Ha(\R_+)$ the operator $f(-\Delta)$,
defined in terms of the functional calculus on $L^2(\R^d)$ can be extended to $L^p(\R^d)$ if
$1 < p < \infty$ and $\alpha > \frac{d}{2}.$
Here
\[ \Ha(\R_+) = \{ f \in C(\R_+,\C) : \sup_{t > 0} \| \phi f(t \cdot) \|_{\Soa(\R_+)} < \infty \} \]
where $\phi \in C^\infty(\R)$ with compact $\supp \phi \subset (0,\infty)$ is a cut-off function and $\Soa(\R_+)$ is the usual Riesz-potential Sobolev space.
For $\alpha \in \N,$ an equivalent norm on $\Ha$ is given by the ``classical'' H\"ormander condition
\[ \sup_{R > 0,\,\beta =0,\ldots,\alpha} \frac1R \int_R^{2 R} |t^\beta D^\beta f(t)|^2 dt < \infty . \]
There is a large literature extending such a spectral multiplier result to more general selfadjoint operators on $L^p(\Omega),$ e.g. for Laplace type operators on manifolds, infinite graphs and fractals (see e.g. \cite{Alex,COSY,Duon,DuOS,KuUh,Ouha} and the references therein).
There are various approaches to the $\Ha$ calculus using
kernel estimates, maximal estimates or square function estimates for the resolvent $(\lambda - A)^{-1},$ the analytic semigroup $e^{-zA}$ generated by $-A$ and their ``boundary'', the wave operators $e^{itA}$, or the imaginary powers $A^{it}$ of $A.$
Relevant are e.g. estimates on operator functions such as ($\alpha > \frac12,\: m > \alpha - \frac12$ are fixed)
\begin{itemize}
\item $T_\theta(t) = A^{\frac12} e^{-e^{i\theta}tA},\quad t \in \R_+,$
\item $R_\theta(t) = A^{\frac12} R(e^{i\theta}t,A),\quad t \in \R_+,$
\item $W(s) = |s|^{-\alpha} A^{-\alpha + \frac12} (e^{isA} - 1)^m,\quad s \in \R,$
\item $I(t) = (1 + |t|)^{-\alpha} A^{it},\quad t \in \R.$
\end{itemize}

Many of these estimates imply or are closely related to square sum estimates of the following form
\begin{equation}\label{Equ Intro R-bounded}
\left\| \left( \sum_i |S_i x_i|^2 \right)^{\frac12} \right\|_{L^p} \leq C_1 \left\| \left( \sum_i |x_i|^2 \right)^{\frac12} \right\|_{L^p}
\end{equation}
where $x_i \in L^p(\Omega)$ and the $S_i$ are members of one of the families listed above
(see e.g. \cite{BoCl,Stem} for an early appearance of this square sum estimate in the context of spectral multiplier theorems).
If $(r_n)$ is a sequence of Rademacher functions on $[0,1]$ one can reformulate \eqref{Equ Intro R-bounded} equivalently as
\begin{equation}\label{Equ Intro R-bounded 2}
\int_0^1 \| \sum_i r_i(\omega) S_i x_i \| d\omega \leq C_2 \int_0^1 \| \sum_i r_i(\omega) x_i \| d \omega.
\end{equation}
This statement makes sense in an arbitrary Banach space $X$ and a set $\tau \subset B(X)$ is called $R$-bounded if \eqref{Equ Intro R-bounded 2} holds for all $S_i \in \tau$ and $x_i \in X.$
Using $R$-boundedness in place of kernel estimates and the holomorphic $H^\infty(\Sigma_\sigma)$ calculus instead of the spectral theorem for selfadjoint operators, one can develop a theory of spectral multiplier theorems for $0$-sectorial operators on Banach spaces (see \cite{Kr,Kr2,Kr3,KrW1,KrW2}).
Again, $R$-bounds for one of the operator families listed above are sufficient to secure $\Ha(\R_+)$ spectral theorems for such operators $A.$
However, neither in this general framework nor in the case of Laplace type operators on an $L^p(\Omega)$ space (see above), one obtains necessary and sufficient conditions in terms of $R$-bounds or kernel estimates.
This is related to the (usually) difficult task of determining the optimal $\alpha$ for the $\Ha(\R_+)$ spectral calculus of a given operator $A.$
Thus the purpose of this paper is to give a characterization of the $\Ha(\R_+)$ spectral multiplier theorem in terms of an $L^2$-averaged $R$-boundedness condition.
More precisely, let $t \in J \mapsto N(t) \in B(X)$ be weakly square integrable on an interval $J.$
Then $(N(t))_{t \in J}$ is called $R[L^2]$-bounded if for $h \in L^2(J)$ with $\|h\|_{L^2(J)} \leq 1$ the strong integrals
\[ N_{h}x = \int_J h(t) N(t) x dt,\quad x \in X \]
define an $R$-bounded subset $\{ N_h :\: \|h\|_{L^2(J)} \leq 1 \}$ of $B(X).$
By $R[L^2(J)](N(t)),$ we denote the $R$-bound of this set.
In a Hilbert space $X,$ $R[L^2(J)]$-boundedness reduces to the simple estimate
\[ \left( \int_J |\spr{N(t)x}{y}|^2 dt \right)^{\frac12} \leq C \|x\|\,\|y\| \text{ for all }x,y \in H.\]
Assume now that $A$ is a $0$-sectorial operator with an $H^\infty(\Sigma_\sigma)$ calculus for some $\sigma \in (0,\pi)$ on a Banach space isomorphic to a subspace of an $L^p(\Omega)$ space with $1 \leq p < \infty$ (or more generally, let $X$ have Pisier's property $(\alpha)$).
Then our main results, Theorems \ref{Thm Characterizations Hoermander calculus} and \ref{Thm Sea Ha} show (among other statements), that the following conditions on the operator function above are essentially equivalent:
\begin{align*}
& \text{1. } A\text{ has an }R\text{-bounded }\Ha\text{ spectral calculus, i.e. } \\
& \{ f(A) :\: \|f\|_{\Ha(\R_+)} \leq 1 \}\text{ is }R\text{-bounded in }B(X).\\
& \text{2. resolvents: }R[L^2(\R_+)](R_\theta(\cdot)) \leq C |\theta|^{-\alpha}\text{ for }\theta \to 0\\
& \text{3. semigroup: }R[L^2(\R_+)](T_\theta(\cdot)) \leq C (\frac{\pi}{2} - |\theta|)^{-\alpha}\text{ for }|\theta| \to \frac{\pi}{2}\\
& \text{4. wave operators: }R[L^2(\R)](W(\cdot)) < \infty\\
& \text{5. imaginary powers: }R[L^2(\R)](I(\cdot)) < \infty.
\end{align*}
The estimates on $R_\theta(\cdot)$ (resp. on $T_\theta(\cdot)$) measure the growth of the resolvent (the analytic semigroup) as we approach the spectrum of $A$ (resp. the ``boundary'' $i \R$ of $\C_+$) on rays in $\C \backslash \R_+$ (in $\C_+$).
Clearly, the $R[L^2(J)]$-boundedness of $I(\cdot)$ measures the polynomial growth of the imaginary powers and $W(\cdot)$ the growth of the regularized wave operators $e^{itA}.$
The latter regularization is necessary since outside Hilbert space the operators $e^{itA}$ are usually unbounded.
The equivalence of these statements shows in particular that estimates of resolvents, the semigroup, wave operators or imaginary powers are all equivalent ways to obtain the boundedness of $f(A)$ for arbitrary $f \in \Ha(\R_+).$

We end this introduction with an overview of the article.
Section \ref{Sec 2 Prelims} contains the background on $\HI$ functional calculus for a sectorial operator $A$, $R$-boundedness as well as the definition of relevant function spaces.
In Section \ref{Sec Hormander classes} we introduce the H\"ormander function spaces and their functional calculus.
In Section \ref{Sec 3 Wave} as a preparation for the proof of Theorem \ref{Thm Characterizations Hoermander calculus}, we relate the wave operators $e^{isA}$ with imaginary powers $A^{it}$ via the Mellin transform.
In Section \ref{Sec 4 Averaged}, we study the notion of averaged $R$-boundedness.
We feel that it is worthwhile to introduce averaged $R$-boundedness also for other function spaces than $L^2(J)$ since these notions appeared already implicitly in the literature and have proven to be quite useful \cite[Proposition 4.1, Remark 4.2]{HyVe}, \cite[Corollary 2.14]{KuWe}, \cite[Corollary 3.19]{HaKu}.
Finally in Section \ref{Sec Main Results} we state the main Theorem \ref{Thm Characterizations Hoermander calculus}, which establishes equivalences between the smaller $\Sea$ functional calculus and averaged $R$-boundedness of the operator families above
(see Section \ref{Sec Hormander classes} for the definition of this function space).
However, most of the classical spectral multipliers
(e.g. $f(\lambda) = \lambda^{it}$) belong to $\Ha \backslash \Sea.$
Therefore we extend in Theorem \ref{Thm Sea Ha} this calculus to $\Ha$ by means of a localization procedure.
Finally, in Section \ref{Sec Remarks}, we indicate how our main results can be transferred to bisectorial and strip-type operators.

\section{Preliminaries}\label{Sec 2 Prelims}

\subsection{$0$-sectorial operators}\label{Subsec A B}

We briefly recall standard notions on $\HI$ calculus.
For $\omega \in (0,\pi)$ we let $\Sigma_\omega = \{ z \in \C \backslash \{ 0 \} :\: | \arg z | < \omega \}$ be the sector around the positive axis of aperture angle $2 \omega.$
We further define $\HI(\Sigma_\omega)$ to be the space of bounded holomorphic functions on $\Sigma_\omega.$
This space is a Banach algebra when equipped with the norm $\|f\|_{\infty,\omega} = \sup_{\lambda \in \Sigma_\omega} |f(\lambda)|.$

A closed operator $A : D(A) \subset X \to X$ is called $\omega$-sectorial, if the spectrum $\sigma(A)$ is contained in $\overline{\Sigma_\omega},$ $R(A)$ is dense in $X$ and
\begin{equation}\label{Equ Def Sectorial}
\text{for all }\theta > \omega\text{ there is a }C_\theta > 0\text{ such that }\|\lambda (\lambda - A)^{-1}\| \leq C_\theta \text{ for all }\lambda \in \overline{\Sigma_\theta}^c.
\end{equation}
Note that $\overline{R(A)} = X$ along with \eqref{Equ Def Sectorial} implies that $A$ is injective.
In the literature, in the definition of sectoriality, the condition $\overline{R(A)} = X$ is sometimes omitted.
Note that if $A$ satisfies the conditions defining $\omega$-sectoriality except $\overline{R(A)} = X$ on $X = L^p(\Omega),\, 1 < p < \infty$ (or any reflexive space),
then there is a canonical decomposition $X  = \overline{R(A)} \oplus N(A),\,x = x_1 \oplus x_2,$ and $A = A_1 \oplus 0,\,x \mapsto A x_1 \oplus 0,$
such that $A_1$ is $\omega$-sectorial on the space $\overline{R(A)}$ with domain $D(A_1) = \overline{R(A)} \cap D(A).$

For an $\omega$-sectorial operator $A$ and a function $f \in \HI(\Sigma_\theta)$ for some $\theta \in (\omega,\pi)$ that satisfies moreover an estimate $|f(\lambda)| \leq C |\lambda|^\epsilon / | 1 + \lambda |^{2\epsilon},$
one defines the operator
\begin{equation}\label{Equ Cauchy Integral Formula}
f(A) = \frac{1}{2 \pi i} \int_{\Gamma} f(\lambda) (\lambda - A)^{-1} d\lambda ,
\end{equation}
where $\Gamma$ is the boundary of a sector $\Sigma_\sigma$ with $\sigma \in (\omega,\theta),$ oriented counterclockwise.
By the estimate of $f,$ the integral converges in norm and defines a bounded operator.
If moreover there is an estimate $\|f(A)\| \leq C \|f\|_{\infty,\theta}$ with $C$ uniform over all such functions, then $A$ is said to have a bounded $\HI(\Sigma_\theta)$ calculus.
In this case, there exists a bounded homomorphism $\HI(\Sigma_\theta) \to B(X),\,f \mapsto f(A)$ extending the Cauchy integral formula \eqref{Equ Cauchy Integral Formula}.

We refer to \cite{CDMY} for details.
We call $A$ $0$-sectorial if $A$ is $\omega$-sectorial for all $\omega > 0.$

For $\omega \in (0,\pi),$ define the algebras of functions $\Hol(\Sigma_\omega) = \{ f : \Sigma_\omega \to \C :\: \exists \: n \in \N :\: \rho^n f \in \HI(\Sigma_\omega) \},$
where $\rho(\lambda) = \lambda (1 + \lambda)^{-2}.$
For a proof of the following lemma, we refer to \cite[Section 15B]{KuWe} and \cite{Haasb},\cite{Haasc}.

\begin{lem}\label{Lem Hol}
Let $A$ be a $0$-sectorial operator.
There exists a linear mapping, called the extended holomorphic calculus,
\begin{equation}\label{Equ Extended HI calculus}
\bigcup_{0 < \omega < \pi} \Hol(\Sigma_\omega) \to \{ \text{closed and densely defined operators on }X \},\: f \mapsto f(A)
\end{equation}
extending \eqref{Equ Cauchy Integral Formula} such that for any $f,g \in \Hol(\Sigma_\omega),$ $f(A)g(A)x = (f g)(A)x$ for $x \in \{y \in D(g(A)):\: g(A)y \in D(f(A)) \} \subset D((fg)(A))$ and
$D(f(A)) = \{ x \in X :\: (\rho^n f)(A) x \in D(\rho(A)^{-n}) = D(A^n) \cap R(A^n) \},$ where $(\rho^n f)(A)$ is given by \eqref{Equ Cauchy Integral Formula}, i.e. $n \in \N$ is sufficiently large.
\end{lem}

\subsection{Function spaces on the line and half-line}\label{Subsec Prelims Function spaces}

In this subsection, we introduce several spaces of differentiable functions on $\R_+ = (0,\infty)$ and $\R.$
Let $\equi \in C^\infty_c(\R).$
Assume that $\supp \equi \subset [-1,1]$ and $\sum_{n=-\infty}^\infty \equi(t-n) = 1$ for all $t \in \R.$
For $n \in \Z,$ we put $\equi_n = \equi(\cdot - n)$ and call $(\equi_n)_{n \in \Z}$ an equidistant partition of unity.
Let $\dyad \in C^\infty_c(\R_+).$
Assume that $\supp \dyad \subset [\frac12,2]$ and $\sum_{n=-\infty}^\infty \dyad(2^{-n} t) =1$ for all $t > 0.$
For $n \in \Z,$ we put $\dyad_n = \dyad(2^{-n} \cdot)$ and call $(\dyad_n)_{n \in \Z}$ a dyadic partition of unity.
Next let $\Fdyad_0,\,\Fdyad_1 \in C^\infty_c(\R)$ such that $\supp \Fdyad_1 \subset [\frac12,2]$ and $\supp \Fdyad_0 \subset [-1,1].$
For $n \geq 2,$ put $\Fdyad_n = \Fdyad_1(2^{1-n}\cdot),$ so that $\supp \Fdyad_n \subset [2^{n-2},2^n].$
For $n \leq -1,$ put $\Fdyad_n = \Fdyad_{-n}(-\cdot).$
We assume that $\sum_{n \in \Z} \Fdyad_n(t) = 1$ for all $t \in \R.$
Then we call $(\Fdyad_n)_{n \in \Z}$ a dyadic partition of unity on $\R$, which we will exclusively use to decompose the Fourier image of a function.
For the existence of such partitions, we refer to the idea in \cite[Lemma 6.1.7]{BeL}.
We recall the following classical function spaces:

\begin{nota}\label{Nota Classical function spaces}
Let $m \in \N_0$ and $\alpha > 0.$
\begin{enumerate}
\item $C^m_b = \{ f : \R \to \C:\: f \,m\text{-times diff. and }f,f',\ldots,f^{(m)}\text{ uniformly cont. and bounded}\}.$
\item $\Soa = \{ f \in L^2(\R):\: \|f\|_{\Soa} = \| (\hat{f}(t)(1 + |t|)^\alpha)\ck\|_2 < \infty \}.$
\end{enumerate}
The space $\Soa$ is a Banach algebra with respect to pointwise multiplication if $\alpha > \frac12,$ \cite[p.~222]{RuSi}.

Further we also consider the local space
\begin{enumerate}
\item[(3)] $\Soaloc = \{ f : \R \to \C :\: f\varphi \in \Soa\text{ for all }\varphi \in C^\infty_c\}$ for $\alpha > \frac12.$
\end{enumerate}
This space is closed under pointwise multiplication.
Indeed, if $\varphi \in C^\infty_c$ is given, choose $\psi \in C^\infty_c$ such that $\psi \varphi = \varphi.$
For $f,g \in \Soaloc,$ we have $(fg)\varphi = (f \varphi)(g \psi) \in \Soaloc.$
\end{nota}

\subsection{Rademachers, Gaussians and $R$-boundedness}\label{Subsec Prelims Rad}

A classical theorem of Marcinkiewicz and Zygmund states that for elements $x_1,\ldots,x_n \in L^p(U,\mu)$ we can express ``square sums'' in terms of random sums
\[ \left\| \left( \sum_{j=1}^n |x_j(\cdot)|^2 \right)^{\frac12} \right\|_{L^p(U)}
\cong \left( \E \| \sum_{j=1}^n \epsilon_j x_j \|_{L^p(U)}^q \right)^{\frac1q}
\cong \left( \E \| \sum_{j=1}^n \gamma_j x_j \|_{L^p(U)}^q \right)^{\frac1q} \]
with constants only depending on $p,q \in [1,\infty).$
Here $(\epsilon_j)_j$ is a sequence of independent Bernoulli random variables (with $P(\epsilon_j = 1) = P(\epsilon_j = -1) =\frac12$) and $(\gamma_j)_j$ is a sequence of independent standard Gaussian random variables.
Following \cite{Bou} it has become standard by now to replace square functions in the theory of Banach space valued function spaces by such random sums (see e.g. \cite{KuWe}).
Note however that Bernoulli sums and Gaussian sums for $x_1,\ldots,x_n$ in a Banach space $X$ are only equivalent if $X$ has finite cotype (see \cite[p.~218]{DiJT} for details).

Let $\tau$ be a subset of $B(X).$
We say that $\tau$ is $R$-bounded if there exists a $C < \infty$ such that
\[ \E \bignorm{ \sum_{k=1}^n \epsilon_k T_k x_k } \leq C \E \bignorm{ \sum_{k=1}^n \epsilon_k x_k } \]
for any $n \in \N,$ $T_1,\ldots, T_n \in \tau$ and $x_1,\ldots,x_n \in X.$
The smallest admissible constant $C$ is denoted by $R(\tau).$
We remark that one always has $R(\tau) \geq \sup_{T \in \tau} \|T\|$ and equality holds if $X$ is a Hilbert space.

Recall that by definition, $X$ has Pisier's property $(\alpha)$ if for any finite family $x_{k,l}$ in $X,$ $(k,l) \in F,$ where $F \subset \Z \times \Z$ is a finite array,
we have a uniform equivalence
\[ \E_\omega \E_{\omega'} \bignorm{ \sum_{(k,l) \in F} \epsilon_k(\omega) \epsilon_l(\omega') x_{k,l} }_{X} \cong \E_\omega \bignorm{ \sum_{(k,l) \in F} \epsilon_{k,l}(\omega) x_{k,l} }_{X}. \]
Note that property $(\alpha)$ is inherited by closed subspaces, and that an $L^p$ space has property $(\alpha)$ provided $1 \leq p < \infty$ \cite[Section 4]{KuWe}.

\section{H\"ormander classes}\label{Sec Hormander classes}

Aside from the classical spaces in Notation \ref{Nota Classical function spaces} we introduce the following H\"ormander class.
We write from now on
\[f_e : J \to \C,\,z \mapsto f(e^z)\]
for a function $f : I \to \C$ such that $I \subset \C \backslash (-\infty,0]$ and $J = \{ z \in \C : \: | \Im z | < \pi,\:e^z \in I \}.$

\begin{defi}\label{Def Mih Hor}~
\begin{enumerate}
\item Let $\alpha > \frac12.$
We define
\[ \Sea = \{ f : (0,\infty) \to \C :\: \|f\|_{\Sea} = \|f_e\|_{\Soa} < \infty \} \]
and equip it with the norm $\|f\|_{\Sea}.$
\item
Let $(\equi_n)_{n \in \Z}$ be an equidistant partition of unity and $\alpha > \frac12.$
We define the H\"ormander class
\[\Ha = \{ f \in L^2_{\text{loc}}(\R_+) :\: \|f\|_{\Ha} = \sup_{n \in \Z} \|\equi_n f_e\|_{\Sob^\alpha} < \infty\}\]
and equip it with the norm $\|f\|_{\Ha}.$
\end{enumerate}
\end{defi}

We have the following elementary properties of H\"ormander spaces.
Its proof may be found in \cite[Propositions 4.8 and 4.9, Remark 4.16]{Kr}.

\begin{lem}\label{Lem Elementary Mih Hor}~
\begin{enumerate}
\item The spaces $\Sea$ and $\Ha$ are Banach algebras.
\item Different partitions of unity $(\equi_n)_n$ give the same space $\Ha$ with equivalent norms.
\item\label{it Embeddings Hormander}
Let $\alpha > \frac12$ and $\sigma \in (0,\pi).$
Then
\[\HI(\Sigma_\sigma) \hookrightarrow \Ha .\]
\item For any $t > 0,$ we have $\|f\|_{\Ha} \cong \|f(t \cdot)\|_{\Ha}.$
\end{enumerate}
\end{lem}

\begin{rem}\label{Rem Classical Mih Hor}
The name ``H\"ormander class'' is justified by the following facts.
The classical H\"{o}rmander condition with a parameter $\alpha_1 \in \N$ reads as follows \cite[(7.9.8)]{Hoa}:
\begin{equation}\label{Equ Classical Hoermander condition}
\sum_{k = 0}^{\alpha_1} \sup_{R > 0} \int_{R/2}^{2R} |R^k f^{(k)}(t)|^2 dt/R < \infty.
\end{equation}
Furthermore, consider the following condition for some $\alpha > \frac12:$
\begin{equation}\label{Equ Modern Hoermander condition}
\sup_{t > 0} \| \psi f(t\cdot) \|_{\Soa} < \infty,
\end{equation}
where $\psi$ is a fixed function in $C^\infty_c(\R_+) \backslash \{ 0 \}.$
This condition appears in several articles on H\"ormander spectral multiplier theorems, we refer to \cite{DuOS} for an overview.
One easily checks that \eqref{Equ Modern Hoermander condition} does not depend on the particular choice of $\psi$ (see also \cite[p. 445]{DuOS}).
\end{rem}

By the following lemma which is proved in \cite[Proposition 4.11]{Kr}, the norm $\|\cdot\|_{\Ha}$ expresses condition \eqref{Equ Modern Hoermander condition}
and generalizes the classical H\"{o}rmander condition \eqref{Equ Classical Hoermander condition}.

\begin{lem}\label{Lem Classical and modern Hoermander condition}
Let $f \in L^1_{\text{loc}}(\R_+),$ $\alpha_1 \in \N$ and $\alpha > \frac12.$
Consider the conditions
\begin{enumerate}
\item $f$ satisfies \eqref{Equ Classical Hoermander condition},
\item $f$ satisfies \eqref{Equ Modern Hoermander condition},
\item $\|f\|_{\Ha} < \infty.$
\end{enumerate}
Then $(1) \Rightarrow (2)$ if $\alpha_1 \geq \alpha$ and $(2) \Rightarrow (1)$ if $\alpha \geq \alpha_1.$
Further, $(2) \Leftrightarrow (3).$
\end{lem}

Let $E$ be a Sobolev space as in Notation \ref{Nota Classical function spaces}.
In this subsection we define an $E$ functional calculus for a $0$-sectorial operator $A$
by tracing it back to the holomorphic functional calculus from Subsection \ref{Subsec A B}.
The following lemma which is proved in \cite[Lemma 4.15]{Kr} will be useful.

\begin{lem}\label{Lem HI dense in diverse spaces}
Let $\beta > \frac12.$
Then $\bigcap_{0 < \omega < \pi}\HI(\Sigma_\omega) \cap \Sobexp^\beta_2$ is dense in $\Sobexp^\beta_2.$
More precisely, if $f \in \Sobexp^\beta_2,$ $\psi \in C^\infty_c$ such that $\psi(t) =1$ for $|t| \leq 1$ and $\psi_n = \psi(2^{-n}(\cdot)),$
then
\[ (f_e \ast \check\psi_n) \circ \log \in \bigcap_{0 < \omega < \pi} \HI(\Sigma_\omega) \cap \Sobexp^\beta_2 \text{ and }(f_e \ast \check\psi_n)\circ\log \to f\text{ in }\Sobexp^\beta_2. \]
Thus if $f$ happens to belong to several spaces $\Sobexp^\beta_2$ for different $\beta$ as above,
then it can be simultaneously approximated by a holomorphic sequence in any of these spaces.
\end{lem}

Lemma \ref{Lem HI dense in diverse spaces} enables to base the $\Sobexp^\beta_2$ calculus on the $\HI$ calculus.

\begin{defi}\label{Def Line calculi}
Let $A$ be a 0-sectorial operator and $\beta > \frac12.$
We say that $A$ has a (bounded) $\Sobexp^\beta_2$ calculus if there exists a constant $C > 0$ such that
\[
\|f(A)\| \leq C \|f\|_{\Sobexp^\beta_2}\quad (f \in \bigcap_{0 < \omega < \pi} \HI(\Sigma_\omega) \cap \Sobexp^\beta_2).
\]
In this case, by the just proved density of $\bigcap_{0 < \omega < \pi} \HI(\Sigma_\omega) \cap \Sobexp^\beta_2$ in $\Sobexp^\beta_2,$ the algebra homomorphism $u : \bigcap_{0 < \omega < \pi} \HI(\Sigma_\omega) \cap \Sobexp^\beta_2 \to B(X)$ given by $u(f) = f(A)$ can be continuously extended in a unique way to a bounded algebra homomorphism
\[u: \Sobexp^\beta_2 \to B(X),\,f \mapsto u(f).\]
We write again $f(A)=u(f)$ for any $f \in \Sobexp^\beta_2.$
Assume that $\beta_1,\beta_2 > \frac12$ and that $A$ has a $\Sobexp^{\beta_1}_2$ calculus and a $\Sobexp^{\beta_2}_2$ calculus.
Then for $f \in \Sobexp^{\beta_1}_2 \cap \Sobexp^{\beta_2}_2,\, f(A)$ is defined twice by the above.
However, the second part of Lemma \ref{Lem HI dense in diverse spaces} shows that these definitions coincide.
\end{defi}

The following lemma gives a representation formula of the $\Sobexp^\alpha_2$ calculus in terms of the $C_0$-group $A^{it}.$
It can be proved with the Cauchy integral formula \eqref{Equ Cauchy Integral Formula} in combination with the Fourier inversion formula
\cite[Proposition 4.22]{Kr}.
Here and below we use the short hand notation $\langle t \rangle = \sqrt{ 1 + t^2 }.$

\begin{lem}\label{Lem Sobolev calculus}
Let $X$ be a Banach space with dual $X'.$
Let $\alpha > \frac12,$ so that $\Soa$ is a Banach algebra.
Let $A$ be a $0$-sectorial operator with bounded imaginary powers $U(t) = A^{it}.$
\begin{enumerate}
\item
Assume that for some $C > 0$ and all $x \in X,\,x'\in X'$
\begin{equation}\label{Equ group weak L2}
 \| \tma \spr{U(t)x}{x'} \|_{L^2(\R)} = \left( \int_\R |\tma \spr{U(t)x}{x'}|^2 dt \right)^{1/2} \leq C \| x \| \, \| x' \|.
\end{equation}
Then $A$ has a bounded $\Sea$ calculus.
Moreover, for any $f \in \Sea,$ $f(A)$ is given by
\begin{equation}\label{Equ Soa calculus formula}
 \spr{f(A) x}{x'} = \frac{1}{2\pi} \int_\R (f_e)\hat{\phantom{i}}(t) \spr{U(t)x}{x'} dt \quad (x \in X,\: x' \in X').
\end{equation}
The above integral exists as a strong integral if moreover $\| \tma U(t) x \|_{L^2(\R)} < \infty.$
\item Conversely, if $A$ has a $\Sea$ calculus, then \eqref{Equ group weak L2} holds.
\end{enumerate}
\end{lem}

\begin{proof}
(1): For $f \in \Sea,x \in X,$ and $x' \in X',$ set
\[\spr{\Phi(f)x}{x'} = \frac{1}{2\pi} \int_\R (f_e)\hat{\phantom{i}}(t) \spr{U(t)x}{x'} dt.\]
We have
\begin{align}
|\spr{\Phi(f)x}{x'}| & \lesssim \int_\R | (f_e)\hat{\phantom{i}}(t) \spr{U(t)x}{x'}| dt = \int_\R |\ta  (f_e)\hat{\phantom{i}}(t) \tma \spr{U(t)x}{x'}| dt
\nonumber \\
& \leq \|\ta  (f_e)\hat{\phantom{i}}(t)\|_{2} \| \tma \spr{U(t)x}{x'} \|_2  \overset{\eqref{Equ group weak L2}}{\lesssim} \|\ta  (f_e)\hat{\phantom{i}}(t)\|_{2}\|x\|\,\|x'\|
\nonumber \\
& \leq \|f\|_{\Sea}\|x\|\,\|x'\|,
\label{Equ 2 Proof Soa calculus}
\end{align}
so that $\Phi$ defines a bounded operator $\Sea \to B(X,X'').$
Let
\[ K = \bigcap_{\omega > 0} \{ f \in \HI(\Sigma_\omega):\: \exists\,C > 0 \: \forall z \in \Sigma_\omega:\: |f_e(z)| \leq C ( 1 + |\Re z| )^{-2}\text{ and } (f_e)\hat{\phantom{i}}\text{ has comp. supp.}\}\]
We have that $K$ is a dense subset of $\Sea.$

Indeed, by the Cauchy integral formula, $K \subset \Sea.$
We now approximate a given $f \in \Sea$ by elements of $K.$
Since $C^\infty_c(0,\infty)$ is dense in $\Sea,$ we can assume $f\in C^\infty_c(0,\infty).$
Let $\psi$ and $\psi_n$ be as in the Density Lemma \ref{Lem HI dense in diverse spaces} and put $f_n = (f_e \ast \check\psi_n) \circ \log.$
Then $(f_{n,e})\hat{\phantom{i}} = (f_e)\hat{\phantom{i}} \psi_n$ has compact support.
Further, the estimate $|f_{n,e}(z)| \leq C (1+|\Re z|)^{-2}$ for $z$ in a given strip $\{ \lambda \in \C: \: |\Im \lambda| < \omega \}$ follows from the Paley-Wiener theorem
and the fact that $(f_{n,e})\hat{\phantom{i}} = (f_e)\hat{\phantom{i}} \psi_n \in C^\infty_c(\R).$
Thus, $(f_e \ast \check\psi_n) \circ \log \in K,$ and $K$ is dense in $\Sea.$

Assume for a moment that
\begin{equation}\label{Equ Proof Soa calculus}
 \Phi(f) = f(A) \quad (f \in K).
\end{equation}

Then by \eqref{Equ 2 Proof Soa calculus}, there exists $C > 0$ such that for any $f \in K,\,\|f(A)\| \lesssim \|f\|_{\Sea}.$
By the density of $K$ in $\Sea,$ $A$ has a $\Sea$ calculus.
Then for any $f \in \Sea$ and $(f_n)_n$ a sequence in $K$ such that $f = \lim_n f_n,$
\[f(A) = \lim_n f_n(A) = \lim_n \Phi(f_n) = \Phi(f),\]
where limits are in $B(X,X'').$
Thus, $f(A) = \Phi(f)$ for arbitrary $f,$ and \eqref{Equ Soa calculus formula} follows.

We show \eqref{Equ Proof Soa calculus}.
Let $f \in K.$
Denote $iB$ the generator of the group $U(t).$
We argue as in \cite[Lemma 2.2]{Haas}.
Choose some $\omega > 0.$
According to the representation formula
\begin{equation}\label{Equ Resolvent Laplace Formula}
R(\lambda,B)x = -\sgn(\Im \lambda)i\int_0^\infty e^{i \sgn(\Im \lambda)\lambda t} U(-\sgn(\Im \lambda) t)x dt,
\end{equation}
we have by the composition rule \cite[Theorem 4.2.4]{Haasa}
\begin{align}
\spr{f(A)x}{x'} & = \frac{1}{2\pi i} \int_\R f_e(s - i\omega) \spr{R(s-i\omega,B)x}{x'} ds
-\frac{1}{2\pi i}\int_\R f_e(s+i\omega) \spr{R(s+i\omega,B)x}{x'} ds
\nonumber \\
& = \frac{1}{2\pi i}\left[ \int_\R f_e(s-i\omega) \cdot i\int_0^\infty e^{-i(s-i\omega) t}\spr{U(t)x}{x'} dt ds\right.
\nonumber \\
& \left.+ \int_\R f_e(s+i\omega) \cdot i \int_{-\infty}^0 e^{-i(s + i\omega) t} \spr{U(t)x}{x'} dt ds\right]
\nonumber \\
& \overset{(*)}{=} \frac{1}{2\pi} \left[ \int_0^\infty \left( \int_\R f_e(s-i\omega) e^{-i(s-i\omega)t} ds\right) \spr{U(t)x}{x'} dt\right.
\nonumber \\
& \left.+ \int_{-\infty}^0 \left( \int_\R f_e(s+i\omega) e^{-i(s+i\omega)t} ds \right) \spr{U(t)x}{x'} dt \right]
\nonumber \\
& \overset{(**)}{=} \frac{1}{2\pi} \left[ \int_0^\infty \int_\R f_e(s) e^{-ist} ds \spr{U(t)x}{x'} dt
+ \int_{-\infty}^0 \int_\R f_e(s) e^{-ist} ds \spr{U(t)x}{x'} dt \right]
\nonumber \\
& = \frac{1}{2\pi} \int_\R (f_e)\hat{\phantom{i}}(t) \spr{U(t)x}{x'} dt
\nonumber \\
& = \spr{\Phi(f)x}{x'}.
\label{Equ 3 Proof Sobolev calculus}
\end{align}
As $f \in K$, we could apply Fubini's theorem in $(*)$ and shift the contour of the integral in $(**).$
Hence \eqref{Equ Proof Soa calculus} follows.

The last sentence of part (1) is now clear.\\

(2): If $A$ has a $\Sea$ calculus, then \eqref{Equ 3 Proof Sobolev calculus} still holds,
and thus
\[ \spr{f(A)x}{x'} = \frac{1}{2\pi}\int_\R (f_e)\hat{\phantom{i}}(t) \spr{U(t)x}{x'} dt \quad (f \in K).\]
Therefore, by the density of $K$ in $\Sea,$
\begin{align}
\| \tma \spr{U(t)x}{x'} \|_2  & = {2\pi} \sup\{|\spr{f(A)x}{x'}|:\:f \in K,\,\|\ta (f_e)\hat{\phantom{i}}(t)\|_2\leq 1\}
\nonumber \\
& \lesssim \|x\|\,\|x'\|.
\nonumber
\end{align}
\end{proof}

In order to state some of our main results for a general class of operators $A,$ we introduce now an auxiliary functional calculus which allows to define $f(A)$ as a closed, not necessarily bounded operator for $f_e$ in (a subclass) of $\Soaloc.$
Let $A$ be a $0$-sectorial operator on some Banach space $X.$
For $\theta > 0,$ we let $D(\theta) = D(A^\theta) \cap R(A^\theta),$ which is a Banach space with the norm $\|x\|_{D(\theta)} = \| \rho^{-\theta}(A)x \|_X.$
$D(\theta)$ forms a decreasing scale of spaces when $\theta$ grows.
Recall that $\rho(\lambda) = \lambda ( 1 + \lambda)^{-2}$ and its powers $\rho^\theta$ belong to $\HI_0(\Sigma_\omega)$ for any $\omega \in (0,\pi),$ and $R(\rho^\theta(A)) = D(A^\theta) \cap R(A^\theta).$
Note that $D(\theta)$ is dense in $X$ (see \cite[9.4 Proposition (c)]{KuWe}).
Assume that $A$ satisfies one of the following assumptions for some $\theta > 0$ and $\beta > 0.$
\begin{align}
 \int_\R |\langle t \rangle^{-\beta} \langle A^{it}x,x' \rangle|^2 dt & \leq C \|x\|_{D(\theta)}^2 \|x'\|_{X'}^2, \label{equ 1 assumptions auxiliary calculus} \\
 \int_\R |\langle s \rangle^{-\beta} \langle e^{isA} x, x' \rangle|^2 ds & \leq C \|x\|_{D(\theta)}^2 \|x'\|_{X'}^2, \label{equ 2 assumptions auxiliary calculus} \\
\int_0^\infty | \langle \exp(-e^{i \omega} t A)x, x' \rangle|^2 dt & \leq C (\frac{\pi}{2} - |\omega|)^{-2\beta} \|x\|_{D(\theta)}^2 \|x'\|_{X'}^2\text{ for }|\omega| < \frac{\pi}{2}, \label{equ 3 assumptions auxiliary calculus} \\
 \int_0^\infty | t^\gamma \langle R(e^{i\omega} t, A) x, x' \rangle |^2 \frac{dt}{t} & \leq C |\omega|^{-2 \beta} \|x\|_{D(\theta)}^2 \|x'\|_{X'}^2\text{ for some fixed }\gamma \in (0,1)\text{ and any }\omega \in (-\pi,\pi) \backslash \{ 0 \}. \label{equ 4 assumptions auxiliary calculus}
\end{align}
Here in \eqref{equ 2 assumptions auxiliary calculus}, $\langle e^{isA}x,x'\rangle$ is understood to be the limit of $\langle e^{(is-t)A}x,x' \rangle$ as $t \to 0+.$
If $A$ satisfies \eqref{equ 1 assumptions auxiliary calculus} with $\beta = \alpha > \frac12$, then one can show as in Lemma \ref{Lem Sobolev calculus} (1) that there is a bounded linear mapping
\[ \Phi_A : \Sea \to B(D(\theta),X),\: f \mapsto \Phi_A(f)\]
with $\langle \Phi_A(f)x, x' \rangle = \frac{1}{2\pi} \int_{\R} (f_e)\hat{\phantom{i}}(t) \langle A^{it}x, x' \rangle dt$ for $x \in D(\theta)$ and $x' \in X'.$
Moreover, $\Phi_A(f)x = f(A)x$ for $x \in D(\theta)$ and $f \in \HI(\Sigma_\omega) \cap \Sea,$ where the right hand side is defined by the holomorphic functional calculus from Lemma \ref{Lem Hol}.
Note that $\Phi_A(f)$ is a closeable operator on $X.$
Indeed, if $x_n \to 0$ in $X$ with $x_n \in D(\theta)$ and $\Phi_A(f)x_n \to y$ for some $y \in X,$ then $\Phi_A(f)\rho^\theta(A) = \rho^\theta(A) \Phi_A(f)$ is a bounded operator in $B(X),$ so that $\rho^\theta(A) \Phi_A(f) x_n$ converges to both $0$ and $\rho^\theta(A)y.$
By injectivity of $\rho^\theta(A),$ it follows that $y = 0,$ and thus, $\Phi_A$ is closeable.
Therefore we can denote without ambiguity its closure by $f(A)$.
We will show later in Proposition \ref{Prop assumptions auxiliary calculus} that the conditions \eqref{equ 2 assumptions auxiliary calculus}, \eqref{equ 3 assumptions auxiliary calculus} and \eqref{equ 4 assumptions auxiliary calculus} imply \eqref{equ 1 assumptions auxiliary calculus} for some $\alpha > \frac12$ and $\theta' > 0$ and thus also imply the auxiliary functional calculus $\Phi_A$ with the properties above.

\begin{lem}
\label{Lem auxiliary calculus multiplicative}
Let $A$ be a $0$-sectorial operator with auxiliary functional calculus $\Phi_A$ as above.
Let $\omega \in (0,\pi).$
\begin{enumerate}
\item If $f \in \Sea ,\: g \in \HI(\Sigma_\omega)$ and $x \in D(\theta),$ then $f(A) g(A)x = (fg)(A)x.$
\item If $f \in \Sea,$ $g \in \HI_0(\Sigma_\omega)$ and $x \in D(\theta),$ then $g(A) f(A)x = (gf)(A) x.$
\item If $f,g \in \Sea$ and $x \in D(\theta),$ then $g(A)x \in D(f(A))$ and $f(A)g(A)x = (fg)(A)x.$
\end{enumerate}
\end{lem}

\begin{proof}
(1) Let $f_n \in \Sea \cap \HI(\Sigma_\omega)$ with $f_n \to f$ in $\Sea.$
Then on the one hand, since $g(A)$ commutes with $\rho^\theta(A),$ and thus, $g(A)x$ belongs to $D(\theta),$ we have $f_n(A)g(A)x \to f(A) g(A) x.$
On the other hand, $f_n(A) g(A) x  = (f_ng)(A) x \to (fg)(A)x,$ since $f_ng \to fg$ in $\Sea.$\\

(2) Let again $f_n \in \Sea \cap \HI(\Sigma_\omega)$ with $f_n \to f$ in $\Sea.$
Then $f_n(A)x \to f(A)x,$ so $g(A)f_n(A) x \to g(A) f(A) x.$
On the other hand, $g(A) f_n(A) x = (g f_n)(A)x \to (gf)(A)x,$ since $gf_n \to gf$ in $\Sea.$\\

(3) 
We first show that if $x \in D(2\theta),$ then $g(A)x \in D(\theta)$ and $f(A)g(A)x = (fg)(A)x.$
Note that by (1) and (2), $g(A)$ commutes with $\rho^\theta(A),$ so that $g(A)$ maps $D(2\theta)$ into $D(\theta).$
This implies $g(A)x \in D(\theta).$
Moreover, we have for $x' \in X',$
\begin{align*}
\langle (fg)(A) x, x' \rangle & = \frac{1}{2 \pi} \int_{\R} (f_e g_e)\hat{\phantom{i}}(t) \langle A^{it} x, x' \rangle dt \\
& = \frac{1}{2\pi} \int_{\R} \frac{1}{2\pi} (f_e)\hat{\phantom{i}} \ast (g_e)\hat{\phantom{i}} (t) \langle A^{it} x, x' \rangle dt \\
& = \frac{1}{(2\pi)^2} \int_\R \int_\R (f_e)\hat{\phantom{i}}(s) (g_e)\hat{\phantom{i}}(t-s) \langle A^{i(t-s)} A^{is}x, x' \rangle dt ds \\
& = \frac{1}{(2 \pi)^2} \int_\R  (f_e)\hat{\phantom{i}}(s) \int_\R (g_e)\hat{\phantom{i}}(t)
\langle A^{it} A^{is}x, x' \rangle dt ds \\
& = \frac{1}{2\pi} \int_R (f_e)\hat{\phantom{i}}(s) \langle g(A) A^{is} x, x' \rangle ds \\
& = \frac{1}{2\pi} \int_R (f_e)\hat{\phantom{i}}(s) \langle A^{is} g(A) x, x' \rangle ds \\
& = \langle f(A) g(A) x , x' \rangle.
\end{align*}
Now for general $x \in D(\theta),$ we let $x_n \in D(2\theta)$ such that $x_n \to x$ in $D(\theta).$
Then $g(A)x_n \to g(A)x$ in $X.$
By the above, $f(A) g(A)x_n = (fg)(A)x_n,$ which converges to $(fg)(A)x$ in $X.$
By closedness of $f(A),$ this implies $g(A)x \in D(f(A))$ and $f(A)g(A)x = (fg)(A)x.$
\end{proof}

For a $0$-sectorial operator $A$ with auxiliary functional calculus $\Phi_A$ as above, we define the following subset $D_A$ of $X.$
Let $(\dyad_n)_{n \in \Z}$ be a dyadic partition of unity and $\theta$ be given as in \eqref{equ 1 assumptions auxiliary calculus}.
\begin{equation}
\label{Equ D}
D_A = \{\sum_{n = -N}^N \dyad_n(A) x:\: N \in \N,\: x \in D(\theta)\}.
\end{equation}
We call $D_A$ the calculus core of $A.$

\begin{lem}\label{Lem D_A dense}
Let $A$ and $D_A$ be as above.
Then $D_A$ is dense in $X.$
\end{lem}

\begin{proof}
For $n \in \Z,$ let $\phi_n(t) = \exp(-2^n t) - \exp(-2^{n+1}t).$
Then by a telescopic sum argument, $\sum_{n \in \Z} \phi_n(A) x = \lim_{t \to 0} e^{-tA}x - \lim_{t \to \infty} e^{-tA}x = x$ for any $x \in X,$ due to the property $N(A) = \{ 0 \}.$
It thus suffices to show that $\sum_{m \in \Z} \dyad_m(A) \phi_n(A)x$ converges to $\phi_n(A)x$ for $x \in D(2 \theta),$ since $\sum_{m = -N}^N \dyad_m(A) \phi_n(A)x$ belongs to $D_A,$ so that then $\phi_n(A)x \in \overline{D_A}$ and by the above $x \in \overline{D_A}.$
Thus, $D(2\theta) \subset \overline{D_A},$ and we conclude since $D(2\theta)$ is dense in $X.$
We note that by Lemma \ref{Lem auxiliary calculus multiplicative}, $\|\dyad_m(A)\phi_n(A)x\| = \|(\dyad_m \phi_n)(A)x\| \lesssim \|\dyad_m \phi_n\|_{\Sea} \lesssim 2^{-|n+m|}.$
Indeed, the last inequality can be seen as follows.
Let $M > \alpha$ be a natural number.
Then
\begin{align*}
 \| \dyad_m \phi_n \|_{\Sea} & = \| \dyad_{m+n} \phi_0 \|_{\Sea} \\
& \lesssim \| \dyad_0(2^{(\cdot)}) \|_{C^M_b[m+n-1,m+n+1]} \\
& = \| \exp(-2^{(\cdot)}) - \exp(-2 \cdot 2^{(\cdot)}) \|_{C^M_b[m+n-1,m+n+1]}. 
\end{align*}
If $m+n \geq 0,$ then this can be estimated by $\leq \|\exp(-2^{(\cdot)})\|_{C^M_b[m+n-1,m+n+1]} + \|\exp(-2 \cdot 2^{(\cdot)})\|_{C^M_b[m+n-1,m+n+1]} \lesssim \exp(-2^{m+n-1}) \lesssim 2^{-(m+n)}.$
If $m+n \leq 0,$ we use that $\exp(-2^{(\cdot)}) - \exp(-2 \cdot 2^{(\cdot)})$ is holomorphic and in absolute value less than $C |2^{(\cdot)}|,$
for $\Re(\cdot) \leq 0.$
Then the above quantity can be estimated by $\leq C 2^{m+n}.$

In all, $\sum_{m \in \Z} \dyad_m(A) \phi_n(A) x$ converges absolutely in $X.$
Therefore, $\sum_{m = -N}^N (\dyad_m \phi_n)(A) x - \phi_n(A) x = \sum_{|m| \geq N + 1} (\dyad_m \phi_n)(A) x \to 0$ in $X,$ and the claim follows.
\end{proof}

As for the $\HI$ calculus, there is an extended $\Sea$ calculus which is defined for $f_e \in \Soaloc$ with $f \rho^\nu = f(\cdot) (\cdot)^\nu ( 1 + (\cdot))^{-2\nu} \in \Sea$ for some $\nu > 0,$ as a counterpart of \eqref{Equ Extended HI calculus}.

\begin{defi}\label{Def Unbounded Besov or Sobolev calculus}
Let $A$ satisfy one of the conditions \eqref{equ 1 assumptions auxiliary calculus} - \eqref{equ 4 assumptions auxiliary calculus}, so that there is an auxiliary calculus $\Phi_A.$
Let $f_e \in \Soaloc$ with $f \rho^\nu \in \Sea$ for some $\nu > 0.$
We define the operator $f(A)$ on $D_A$ by
\[ f(A) (\sum_{n = -N}^N \dyad_n(A)x) = \sum_{n = -N}^N (f \dyad_n)(A)x.\]
Note that this definition does not depend on the representation $\sum_{n = -N}^N \dyad_n(A) x$ of the element in $D_A.$
\end{defi}

\begin{lem}\label{Lem Soaloc calculus}
Let $A$ and $f$ be as in Definition \ref{Def Unbounded Besov or Sobolev calculus} and $g$ a further function with same assumptions as $f.$
\begin{enumerate}
\item[(a)] The operator $f(A)$ is closeable in $X$, we denote the closure by slight abuse of notation again by $f(A).$
\item[(b)] If furthermore $f \in \Sea$ then $f(A)$ coincides with the operator defined by the calculus $\Phi_A.$
If $f \in \Hol(\Sigma_\omega)$ for some $\omega \in (0, \pi),$ then
$f(A)$ coincides with the (unbounded) holomorphic calculus of $A.$
\item[(c)] For any $x \in D_A,$ we have $g(A)x \in D(f(A))$ and $f(A)g(A)x = (fg)(A)x.$
\end{enumerate}
\end{lem}

\begin{proof}
(a) Let $x_n \in D_A$ with $x_n \to 0$ in $X$ such that $f(A)x_n \to y$ for some $y \in X.$
It is easy to check that $\rho^\nu(A) f(A)x_n = (f \rho^\nu)(A)x_n.$
Then $\rho^{\theta + \nu}(A)f(A)x_n = \rho^\theta(A) (f \rho^\nu)(A) x_n$ converges to $0$ on the one hand, and to $\rho^\theta(A) \rho^\nu(A) y$ on the other hand.
By injectivity of $\rho^{\theta + \nu}(A),$ it follows that $y = 0,$ and thus, $f(A)$ is closeable.\\

(b) For the statement for $f \in \Sea,$ this is easy to check on $D_A.$
Moreover, $D_A$ is dense in $D(f(A))$ by (a) and also in $D(\Phi_A(f))$ which is checked as in Lemma \ref{Lem D_A dense}.
Now pass to the closures of $f(A)$ and $\Phi_A(f).$

For the statement for $f \in \Hol(\Sigma_\omega),$ argue similarly with $D_A$ replaced by $\{ \sum_{n = -N}^N \dyad_n(A)x :\: N \in \N,\: x \in D(\theta + \mu) \}$ and $\mu > 0$ such that $f\rho^\mu \in H^\infty_0(\Sigma_\omega).$\\

(c) We first check that for $x \in D(\theta)$ and $n \in \Z,$ $g(A)\dyad_n(A)x$ belongs to $D(f(A))$ and $f(A)g(A)\dyad_n(A)x = (fg \dyad_n)(A)x.$
To this end, let $x_m \in D(2\theta)$ with $x_m \to x$ in $D(\theta)$ as $m \to \infty.$
Then $\dyad_n(A)x_m \to \dyad_n(A)x$ in $X.$
Moreover, $g(A) \dyad_n(A) x_m = (g \dyad_n)(A)x_m \to (g\dyad_n)(A)x.$
By Lemma \ref{Lem auxiliary calculus multiplicative}, we have $(g \dyad_n)(A) x_m = \dyad_n(A) (g\tdyad_n)(A)x_m,$ where $\tdyad_n = \dyad_{n-1} + \dyad_n + \dyad_{n+1}$ satisfies $\tdyad_n \dyad_n = \dyad_n.$
Since $(g \tdyad_n)(A) x_m \in D(\theta),$ $(g\dyad_n)(A)x_m$ belongs to $D_A.$
By Lemma \ref{Lem auxiliary calculus multiplicative}, $f(A) (g\dyad_n)(A)x_m = f(A) \dyad_n(A) (g \tdyad_n)(A)x_m = (f \dyad_n)(A) (g \tdyad_n)(A) x_m = (fg\dyad_n)(A) x_m \to (fg\dyad_n)(A)x$ in $X.$
By closedness of $f(A),$ $g(A)\dyad_n(A)x = (g \dyad_n)(A)x \in D(f(A))$ and $f(A) g(A) \dyad_n(A) x = (fg \dyad_n)(A) x$ for any $x \in D(\theta).$
We infer
\[f(A)g(A) \sum_{n = -N}^N \dyad_n(A) x = \sum_{n = -N}^N (fg\dyad_n)(A)x = (fg)(A) \sum_{n = -N}^N \dyad_n(A) x \]
for any element $\sum_{n= -N}^N \dyad_n(A) x \in D_A.$
\end{proof}

Note that the H\"ormander class $\Ha$ is contained in $\Sobexp^\alpha_{2,\loc}.$
Thus the $\Sobexp^\alpha_{2,\loc}$ calculus in Lemma \ref{Lem Soaloc calculus} enables us to define the $\Ha$ calculus,
whose boundedness is a main object of investigation in this article.

\begin{defi}\label{Def Hor calculus}
Let $\alpha > \frac12$ and let $A$ be a $0$-sectorial operator having a bounded $\HI(\Sigma_\omega)$ calculus for some $\omega \in (0,\pi).$
We say that $A$ has a (bounded) $\Ha$ calculus if there exists a constant $C > 0$ such that
\begin{equation}\label{Equ Def Hor sect calculus}
\|f(A)\| \leq C \|f\|_{\Ha} \quad (f \in \bigcap_{\omega \in (0,\pi)} \HI(\Sigma_\omega) \cap \Ha).
\end{equation}
\end{defi}

Let $\alpha > \frac12$ and consider a $0$-sectorial operator $A$ having a $\Ha$ calculus in the sense of Definition \ref{Def Hor calculus}.
Then $A$ has a $\Sea$ calculus and thus, we can apply Lemma \ref{Lem Soaloc calculus} (with $\theta = 0$) and consider the unbounded $\Sobexp^\alpha_{2,\loc}$ calculus of $A,$ and in particular $f(A)$ is defined for $f \in \Ha \subset \Sobexp^\alpha_{2,\loc}.$
Then condition \eqref{Equ Def Hor sect calculus} extends automatically to all $f \in \Ha.$

\section{Wave Operators and Bounded Imaginary Powers}\label{Sec 3 Wave}

In this section, we assume that $A$ is a $0$-sectorial operator.
We relate wave operators with imaginary powers of $A$ by means of the Mellin transform
$M : L^2(\R_+,\frac{ds}{s}) \to L^2(\R,dt),\,f \mapsto \int_0^\infty f(s) s^{it} \frac{ds}{s}.$

\begin{prop}\label{Prop Wave BIP Function}
Let $\alpha > \frac12$ and $m \in \N$ such that $m > \alpha - \frac12.$
Assume that $A$ satisfies one of the assumptions \eqref{equ 1 assumptions auxiliary calculus}, \eqref{equ 2 assumptions auxiliary calculus}, \eqref{equ 3 assumptions auxiliary calculus} or \eqref{equ 4 assumptions auxiliary calculus} for some $\beta$ and $\theta.$
Since $\lambda \mapsto \lambda^{\frac12 - \alpha} (e^{\mp is\lambda} - 1)^m$ belongs to $\Sobexp^\alpha_{2,\loc}$ with polynomial growth at $0$ and $\infty,$ this entails that $A^{\frac12 - \alpha} (e^{\mp isA} - 1)^m$ are well-defined closed operators for $s > 0$ with domain containing $D_A$ from \eqref{Equ D}.
Assume moreover that
\[\| t \mapsto \langle t \rangle^{-\alpha} \spr{A^{it}x}{x'} \|_{L^2(\R,dt)} \leq C \|x\|_X \,\|x'\|_{X'} \quad (x \in D_A, \: x' \in X')\] or
\[\| s \mapsto \spr{(sA)^{\frac12 - \alpha} (e^{\mp i s A} - 1)^m x}{x'} \|_{L^2(\R_+, \frac{ds}{s})} \leq C \|x\|_X \, \|x'\|_{X'} \quad (x \in D_A, \: x' \in X').\]
Then for any $x \in D_A$ from \eqref{Equ D}, and $x' \in X',$ we have the identity in $L^2(\R,dt):$
\[ M \left[ \spr{(sA)^{\frac12 - \alpha} (e^{\mp i s A} - 1)^m x}{x'} \right](t) = h_{\mp}(t) \spr{A^{-it}x}{x'},\]
where
\begin{equation}\label{Equ h_mp}
h_\mp(t) = e^{\mp i \frac{\pi}{2} (\frac12 - \alpha)} e^{\pm \frac{\pi}{2} t} \Gamma(\frac12 - \alpha + it) f_m(\frac12 - \alpha + it)
\end{equation}
with
\begin{equation}\label{Equ f_m}
f_m(z) = \sum_{k = 1}^m \binom{m}{k} (-1)^{m-k} k^{-z}
\end{equation}
and $h_\mp$ satisfies $|h_\mp(t)| \lesssim \langle t \rangle^{-\alpha}.$
\end{prop}

The two following lemmas are devoted to the proof of Proposition \ref{Prop Wave BIP Function}.

\begin{lem}\label{Lem Wave Technical 1}
Let $m \in \N$ and $\Re z \in (-m,0).$
Then
\[ \int_0^\infty s^z (e^{-s} - 1)^m \frac{ds}{s} = \Gamma(z) f_m(z),\]
with $f_m$ given in \eqref{Equ f_m}.
Note that $\Gamma(z) f_m(z)$ is a holomorphic function for $\Re z \in (-m,0).$
\end{lem}

\begin{proof}
We proceed by induction over $m.$
In the case $m = 1,$ we obtain by integration by parts
$\int_0^\infty s^{z} (e^{-s}-1) \frac{ds}s
= \left[ \frac1z s^z (e^{-s}-1)\right]^\infty_0 + \int_0^\infty \frac1z s^{z} e^{-s} ds
= 0 + \frac1z \Gamma(z+1) = \Gamma(z) = \Gamma(z)f_1(z).$
Next we claim that for $\Re z > -m,$
\[ \int_0^\infty s^z (e^{-s}-1)^m e^{-s} \frac{ds}s = \Gamma(z) \sum_{k=0}^m \binom{m}{k} (-1)^{m-k} (k+1)^{-z} .\]
Note that the left hand side is well-defined and holomorphic for $\Re z > -m$ and the right hand side is meromorphic on $\C.$
By the identity theorem for meromorphic functions, it suffices to show the claim for e.g. $\Re z > 0.$
For these $z$ in turn, we can develop
\[ \int_0^\infty s^{z} (e^{-s}-1)^m e^{-s} \frac{ds}s = \sum_{k=0}^m \binom{m}{k} (-1)^{m-k} \int_0^\infty s^z e^{-ks}e^{-s} \frac{ds}s,\]
which gives the claim.

Assume now that the lemma holds for some $m.$
Let first $\Re z \in (-m,0).$
In the following calculation, we use both the claim and the induction hypothesis in the second equality,
and the convention $\binom{m}{m + 1} = 0$ in the third.
\begin{align}
\int_0^\infty s^z (e^{-s}-1)^{m+1}\frac{ds}s & = \int_0^\infty s^z (e^{-s}-1)^m e^{-s} \frac{ds}s - \int_0^\infty s^z (e^{-s}-1)^m \frac{ds}s
\nonumber\\
& = \Gamma(z) \sum_{k=0}^m \binom{m}{k} (-1)^{m-k} (k+1)^{-z}  - \Gamma(z) f_m(z)
\nonumber\\
& = \Gamma(z) \sum_{k=1}^{m+1} \binom{m}{k-1} (-1)^{m+1-k} k^{-z}  + \Gamma(z) \sum_{k=1}^{m+1} \binom{m}{k}(-1)^{m+1-k} k^{-z}
\nonumber\\
& = \Gamma(z) \sum_{k=1}^{m+1} \left[ \binom{m}{k-1} + \binom{m}{k}\right](-1)^{m+1-k} k^{-z}
\nonumber\\
& = \Gamma(z) f_{m+1}(z).
\nonumber
\end{align}
Thus, the lemma holds for $m+1$ and $\Re z \in (-m,0).$
For $\Re z \in (-(m+1),-m],$ we appeal again to the identity theorem.
\end{proof}

\begin{lem}\label{Lem Wave Technical 2}
Let $\Re z \in (-m,0)$ and $\Re \lambda \geq 0.$
Then
\[ \int_0^\infty s^z (e^{-\lambda s} - 1)^m \frac{ds}s = \lambda^{-z}\int_0^\infty s^z (e^{-s}-1)^m \frac{ds}s.\]
\end{lem}

\begin{proof}
This is an easy consequence of the Cauchy integral theorem.
\end{proof}

\begin{proof}[Proof of Proposition \ref{Prop Wave BIP Function}]
Let $\mu > 0$ fixed.
Combining Lemmas \ref{Lem Wave Technical 1} and \ref{Lem Wave Technical 2} with $\lambda = \pm i \mu,$ we get
$\int_0^\infty s^z (e^{\mp i \mu s} - 1)^m \frac{ds}{s} = (e^{\pm i \frac{\pi}{2}} \mu)^{-z} \Gamma(z) f_m(z).$
Put now $z = \frac12 - \alpha + it$ for $t \in \R,$ so that $\Re z \in (-m , 0)$ by the assumptions of the proposition.
Then
$\int_0^\infty s^{it + \frac12 - \alpha}(e^{\mp i \mu s} - 1)^m \frac{ds}{s} = e^{\mp i \frac{\pi}{2} (\frac12 - \alpha + it)} \mu^{-it} \mu^{-(\frac12 - \alpha)} \Gamma(z) f_m(z),$
so that with $h_{\mp}(t)$ as in \eqref{Equ h_mp},
\begin{equation}\label{Equ 1 Proof Prop Wave BIP Function} M \left[ (s\mu)^{\frac12 - \alpha} (e^{\mp i s \mu} - 1)^m \right](t) = h_{\mp}(t) \mu^{-it}.
\end{equation}
The statement of the proposition was \eqref{Equ 1 Proof Prop Wave BIP Function} with $\mu$ formally replaced by $A$ (weak identity).
It is easy to see that $\sup_{t \in \R} \left| f_m(\frac12 - \alpha + it) \right| < \infty.$
Further, the Euler Gamma function has a development \cite[p.~15]{Lebe},
$\left|\Gamma(-\frac12 + \alpha + it)\right| \cong e^{-\frac{\pi}{2}|t|} |t|^{-\alpha} \: (|t| \geq 1),$ so that $| h_{\mp}(t) | \lesssim \langle t \rangle^{-\alpha}.$
Thus by the assumption of the proposition, we have $t \mapsto h_{\mp}(t) \spr{A^{-it} x}{x'} \in L^2(\R,dt),$ or $s \mapsto \spr{ (s A)^{\frac12 - \alpha} (e^{\mp i s A} - 1)^m x}{x'} \in L^2(\R_+,\frac{ds}{s}).$
It remains to show that in \eqref{Equ 1 Proof Prop Wave BIP Function}, one can replace $\mu$ by $A$ in the weak sense.
To this end, let $f$ belong to the class $D = \vect\{ t^n e^{-t^2/2} :\: n \in \N_0\},$ which is dense in $L^2(\R).$
Moreover, we shall use that $\max(s^N,s^{-N})M'f(s)\to 0$ for any $N \in \N$ as $s \to 0$ and $s \to \infty$ for such functions $f$ and $M'$ the adjoint mapping of the Mellin transform.
Denote $g(s,\lambda) = (s\lambda)^{\frac12 - \alpha} (e^{\mp is \lambda} - 1)^m.$
Let $y = \sum_{n= -N}^N \dyad_n(A)x \in D_A$ and $x' \in X',$ and write $\tilde{g}(s,\lambda) = g(s,\lambda) \sum_{n = -N}^N \dyad_n(\lambda).$
We have with $F\tilde{g}(s,\lambda)$ denoting the Mellin transform of $\tilde{g}$ in the variable $\lambda$ and $M'$ acting in the variable $s,$
\begin{align*}
\int_\R M[\langle g((\cdot),A)y,x' \rangle](t) f(t) dt & = \int_{\R_+} \langle g(s,A)y,x' \rangle M'f(s) \frac{ds}{s} \\
& = \frac{1}{2\pi} \int_{\R_+} \int_\R F\tilde{g}(s,\lambda) \langle A^{i\lambda} x, x' \rangle d\lambda M'f(s) \frac{ds}{s} \\
& = \frac{1}{2\pi} \int_{\R} \int_{\R_+} F\tilde{g}(s,\lambda) \langle A^{i\lambda} x, x' \rangle M'f(s) \frac{ds}{s} d\lambda \\ 
& = \frac{1}{2\pi} \int_{\R} \int_{\R} MF\tilde{g}(t,\lambda) \langle A^{i\lambda} x, x' \rangle f(t) dt d\lambda \\
& = \frac{1}{2\pi} \int_{\R} \int_{\R} FM\tilde{g}(t,\lambda) \langle A^{i\lambda} x, x' \rangle f(t) dt d\lambda \\
& = \frac{1}{2\pi} \int_{\R} \int_{\R} FM\tilde{g}(t,\lambda) \langle A^{i\lambda} x, x' \rangle d\lambda f(t) dt\\
& = \int_\R \langle (Mg)(t,A) y, x' \rangle f(t) dt.
\end{align*}
Hereby, the (at most) polynomial growth of $\|\tilde{g}(s,\cdot)\|_{\Sea}$ at $s \to 0$ and $s \to \infty,$ the (at most) polynomial growth of $\|M\tilde{g}(t,\cdot)\|_{\Sea}$ at $t \to -\infty$ and $t \to \infty,$ and the decay of $f$ and $M'f$ allowed us to apply two times Fubini's theorem.
By density of such functions $f$ in $L^2(\R),$ the proposition follows.
\end{proof}

A variant of the wave operator expression $(sA)^{\frac12 - \alpha} (e^{\mp isA} - 1)^m$ from Proposition \ref{Prop Wave BIP Function} is given by the following proposition.

\begin{prop}\label{Prop Wave variant}
Let $A$ be a $0$-sectorial operator.
Assume that $A$ satisfies one of the assumptions \eqref{equ 1 assumptions auxiliary calculus}, \eqref{equ 2 assumptions auxiliary calculus}, \eqref{equ 3 assumptions auxiliary calculus} or \eqref{equ 4 assumptions auxiliary calculus} for some $\beta$ and $\theta.$
Let $\alpha - \frac12 \not\in \N_0,$ let $m \in \N_0$ such that $\alpha - \frac12 \in (m,m+1)$ and 
\[ w_\alpha(s) = s^{-\alpha} \left( e^{is} - \sum_{j=0}^{m-1} \frac{(is)^j}{j!} \right). \]
Then $w_\alpha(sA)$ is a closed densely defined operator, and with $M$ denoting again the Mellin transform, we have
\[ M(\langle (sA)^{\frac12} w_\alpha(sA) x, x' \rangle)(t) = i^{-\alpha + \frac12 + it} \Gamma(-\alpha + \frac12 + it) \langle A^{-it} x, x' \rangle \quad (x \in D_A,\: x' \in X'). \]
\end{prop}

\begin{proof}
The proof is similar to that of Proposition \ref{Prop Wave BIP Function}.
We determine the Mellin transform of $s^{\frac12} w_\alpha(s):$
By a contour shift of the integral $s \leadsto is,$
\begin{align*}
 \int_0^\infty s^{it} s^{\frac12} w_\alpha(s) \frac{ds}{s} & = \int_0^\infty (is)^{it} (is)^{\frac12} w_\alpha(is) \frac{ds}{s}
\\ & = i^{-\alpha+\frac12 + it} \int_0^\infty s^{-\alpha+\frac12 + it} (e^{-s} - \sum_{j=0}^{m-1} \frac{(-s)^j}{j!}) \frac{ds}{s}.
\end{align*}
Applying partial integration, one sees that this expression equals $i^{-\alpha + \frac12 + it} \Gamma(-\alpha + \frac12 + it).$
Thus,
\[ M(s^{\frac12} w_\alpha(s))(t) = i^{-\alpha + \frac12 + it} \Gamma(-\alpha + \frac12 + it),\]
and applying the functional calculus yields the proposition, see the end of the proof of Proposition \ref{Prop Wave BIP Function} for details.
\end{proof}

\section{Averaged $R$-boundedness}\label{Sec 4 Averaged}

Let $(\Omega,\mu)$ be a $\sigma$-finite measure space.
Throughout the section, we consider spaces $E$ which are subspaces of the space $\mathcal{L}$ of equivalence classes of measurable functions on $(\Omega,\mu).$
Here, equivalence classes refer to identity modulo $\mu$-null sets.
We require that a subspace $E_0'$ of the dual $E'$ of $E$ is given by
\[ E_0' = \{ f \in \mathcal{L} :\: \exists C > 0 \:\forall \: g \in E: \: |\spr{f}{g}| = |\int_\Omega f(t) g(t) d\mu(t)| \leq C \|g\|_E\} \]
with duality bracket $\langle f, g \rangle = \int_\Omega f(t)g(t) d\mu(t)$ and that this space is norming for $E,$ i.e. $\|g\|_E \cong \sup_{\|f\|_{E'} \leq 1,\:f \in E_0'} |\langle f, g \rangle|.$
This is clearly the case in the following examples:
\begin{align}
E & = L^p(\Omega,w d\mu)\text{ for }1 \leq p \leq \infty\text{ and a weight }w, \nonumber \\
E & = \Soa = \Soa(\R)\text{ for }\alpha > \frac12, \label{Equ Function space E} \\
E & = \Sea.\nonumber
\end{align}

\begin{defi}\label{Def R[E]-bounded}
Let $(\Omega,\mu)$ be a $\sigma$-finite measure space.
Let $E$ be a function space on $(\Omega,\mu)$ as in \eqref{Equ Function space E}.
Let $(N(t):\:t \in \Omega)$ be a family of closed operators on a Banach space $X$ such that
\begin{enumerate}
\item There exists a dense subspace $D_N \subset X$ which is contained in the domain of $N(t)$ for any $t \in \Omega.$
\item For any $x \in D_N,$ the mapping $\Omega \to X,\,t\mapsto N(t) x$ is measurable.
\item For any $x \in D_N,\,x'\in X'$ and $f \in E,\,t \mapsto f(t) \spr{N(t)x}{x'}$ belongs to $L^1(\Omega).$
\end{enumerate}
Then $(N(t):\:t\in \Omega)$ is called $R$-bounded on the $E$-average or $R[E]$-bounded, if for any $f \in E,$
there exists $N_f \in B(X)$ such that
\begin{equation}\label{Equ Def R[E]-bounded}
 \spr{N_f x}{x'} = \int_\Omega f(t) \spr{N(t)x}{x'} d\mu(t) \quad (x\in D_N,\,x'\in X')
\end{equation}
and further
\[R[E](N(t):\: t\in\Omega) := R(\{N_f :\: \|f\|_E \leq 1\}) < \infty.\]
\end{defi}

A number of very useful criteria for $R$-bounded sets known in the literature can be restated in terms of $R[E]$-boundedness.

\begin{exa}\label{Exa R[E]-bounded}
Let $(\Omega,\mu)$ be a $\sigma$-finite measure space and let $(N(t):\:t\in\Omega)$ be a family of closed operators on $X$
satisfying (1) and (2) of Definition \ref{Def R[E]-bounded}.

\begin{enumerate}
\item[a) ($E = L^1$)] Assume that the $N(t)$ are bounded operators.
If $\{N(t):\:t\in\Omega\}$ is $R$-bounded in $B(X)$, then it is also $R[L^1(\Omega)]$-bounded, and
\[ R[L^1(\Omega)](N(t):\:t \in \Omega) \leq 2 R(\{N(t):\: t\in \Omega\}).\]

Conversely, assume in addition that
$\Omega$ is a metric space,  $\mu$ is a $\sigma$-finite strictly positive Borel measure and $t \mapsto N(t)$ is strongly continuous.
If $(N(t):\:t \in \Omega)$ is $R[L^1(\Omega)]$-bounded, then it is also $R$-bounded.
\item[b) ($E = L^\infty$)] Assume that there exists $C > 0$ such that
\[\int_\Omega \|N(t)x\| d\mu(t) \leq C \|x\| \quad (x\in D_N).\]
Then $(N(t):\:t\in\Omega)$ is $R[L^\infty(\Omega)]$-bounded with constant at most $2C.$
\item[c) ($E = L^2$)] Assume that $X$ is a reflexive $L^p(U)$ space.
If \[\left\|\left( \int_\Omega |(N(t)x)(\cdot)|^2 dt\right)^{\frac12} \right\|_{L^p(U)} \leq C \|x\|_{L^p(U)}\] for all $x \in D_N,$
then $(N(t):\:t\in\Omega)$ is $R[L^2(\Omega)]$-bounded and there exists a constant $C_0 = C_0(X)$ such that
\[ R[L^2(\Omega)](N(t):\:t \in \Omega) \leq C_0 C.\]
This can be generalized to spaces $X$ with property $(\alpha)$ and the generalized square function spaces $l(\Omega,X)$ from \cite{KaW2}.
\item[d) ($E = L^{r'}$)] Assume that $X$ has type $p \in [1,2]$ and cotype $q \in [2,\infty].$
Let $1 \leq r,r' < \infty$ with $\frac1r = 1-\frac{1}{r'} > \frac1p - \frac1q.$

Assume that $N(t) \in B(X)$ for all $t \in \Omega,$ that $t \mapsto N(t)$ is strongly measurable, and that
\[\|N(t)\|_{B(X)} \in L^r(\Omega).\]
Then $(N(t):\:t\in\Omega)$ is $R[L^{r'}(\Omega)]$-bounded and there exists a constant $C_0=C_0(r,p,q,X)$ such that
\[ R[L^{r'}(\Omega)](N(t):\:t \in \Omega) \leq C_0 C.\]
\end{enumerate}
\end{exa}

\begin{proof}
($E = L^1$)
Assume that $(N(t):\:t \in \Omega)$ is $R$-bounded.
Then it follows from the Convex Hull Lemma \cite[Lemma 3.2]{CPSW} that $R[L^1(\Omega)](N(t):\:t\in \Omega) \leq 2 R(\{N(t):\:t\in\Omega\}).$
Let us show the converse under the mentioned additional hypotheses.
Suppose that $R(\{N(t):\: t\in \Omega\}) = \infty.$
We will deduce that also $R[L^1(\Omega)](N(t):\:t \in \Omega) = \infty.$
Choose for a given $N \in \N$ some $x_1,\ldots,x_n \in X \backslash \{0\}$ and $t_1,\ldots,t_n \in \Omega$ such that
\[ \E \Bignorm{ \sum_k \epsilon_k N(t_k) x_k }_{X} > N \E \Bignorm{\sum_k \epsilon_k x_k}_{X}. \]
It suffices to show that
\begin{equation}\label{Equ Aux 1}
\E \Bignorm{ \sum_k \epsilon_k \int_\Omega f_k(t)N(t) x_k d\mu(t)}_{X} > N \E \Bignorm{\sum_k \epsilon_k x_k}_{X}
\end{equation}
for appropriate $f_1,\ldots,f_n.$
It is easy to see that by the strong continuity of $N$, \eqref{Equ Aux 1} holds with
$f_k = \frac{1}{\mu(B(t_k,\epsilon))}\chi_{B(t_k,\epsilon)}$ for $\epsilon$ small enough.
Here the fact that $\mu$ is strictly positive and $\sigma$-finite guarantees that $\mu(B(t_k,\epsilon)) \in (0,\infty)$ for small $\epsilon.$\\

\noindent
($E = L^\infty$)
By \cite[the proof of Corollary 2.17]{KuWe} with $Y = X$ there,
\[\E \Bignorm{ \sum_{k=1}^n \epsilon_k N_{f_k} x_k}_{X} \leq 2C \E \Bignorm{ \sum_{k=1}^n \epsilon_k x_k}_{X}\]
for any finite family $N_{f_1},\ldots,N_{f_n}$ from \eqref{Equ Def R[E]-bounded} such that $\|f_k\|_\infty \leq 1,$ and any finite family
$x_1,\ldots,x_n \in D_N.$
Since $D_N$ is a dense subspace of $X,$ we can deduce that $\{N_f:\:\|f\|_\infty \leq 1\}$ is $R$-bounded.\\

\noindent
($E = L^2$)
For $x \in D_N,$ set $\varphi(x) = N(\cdot)x \in L^p(U,L^2(\Omega)).$
By assumption, $\varphi$ extends to a bounded operator $L^p(U) \to L^p(U,L^2(\Omega)).$
Then the assertion follows at once from \cite[Proposition 3.3]{LM} in the case that $\Omega$ is an interval.
The general case that $\Omega$ is a measure space and $X$ has property $(\alpha)$ follows from \cite[Corollary 3.19]{HaKu}.\\

\noindent
($E= L^{r'}$)
This is a result of Hyt\"onen and Veraar, see \cite[Proposition 4.1, Remark 4.2]{HyVe}.
\end{proof}

\begin{prop}\label{Prop General E}
If $E$ is a space as in \eqref{Equ Function space E} and $R[E](N(t):\:t\in\Omega) = C < \infty,$ then
\begin{equation}\label{Equ R[E]-bounded general E}
\|\spr{N(\cdot)x}{x'}\|_{E'} \leq C \|x\|\,\|x'\|\quad (x\in D_N,\,x'\in X').
\end{equation}
In particular, if $1 \leq p,p' \leq \infty$ are conjugated exponents and
\[R[L^{p'}(\Omega)](N(t):\:t\in\Omega) = C < \infty,\] then
\[\left(\int_\Omega | \spr{N(t)x}{x'}|^p d\mu(t)\right)^{1/p} \leq C \|x\|\,\|x'\|\quad (x\in D_N,\,x' \in X').\]
If $X$ is a Hilbert space, then also the converse holds:
Condition \eqref{Equ R[E]-bounded general E} implies that $(N(t):\:t \in \Omega)$ is $R[E]$-bounded.
\end{prop}

\begin{proof}
We have
\begin{align}
& R[E](N(t):\:t\in\Omega)
\nonumber \\
& \geq \sup\{ \|N_f\|_{B(X)}:\:\|f\|_{E} \leq 1\}
\label{Equ Proof Prop Exa R[E]-bounded} \\
& = \sup\left\{ \left | \int_\Omega f(t) \spr{N(t)x}{x'} d\mu(t) \right|:\: \|f\|_{E} \leq 1 ,\,x\in D_N,\|x\| \leq 1, \,x'\in X',\|x'\| \leq 1\right\}
\nonumber \\
& = \sup\left\{\|\spr{N(\cdot)x}{x'}\|_{E'}:\:x \in D_N,\|x\| \leq 1,\,x' \in X',\|x'\| \leq 1\right\}.
\nonumber
\end{align}

If $X$ is a Hilbert space, then bounded subsets of $B(X)$ are $R$-bounded,
and thus, ``$\geq$'' in \eqref{Equ Proof Prop Exa R[E]-bounded} is in fact ``$=$''.
\end{proof}

An $R[E]$-bounded family yields a new averaged $R$-bounded family under a linear transformation in the function space variable.

\begin{lem}\label{Lem Transformation R[E]-bounded}
For $i=1,2,$ let $(\Omega_i,\mu_i)$ be a $\sigma$-finite measure space and $E_i$ a function space on $\Omega_i$ as in \eqref{Equ Function space E},
and $K \in B(E_1',E_2')$ such that its adjoint $K'$ maps $E_2$ to $E_1.$

Let further $(N(t):\:t \in \Omega_1)$ be an $R[E_1]$-bounded family of closed operators and $D_N$ be a core for all $N(t).$
Assume that there exists a family $(M(t):\:t \in \Omega_2)$ of closed operators with the same common core $D_M=D_N$ such that
$t \mapsto M(t)x$ is measurable for all $x \in D_N$ and
\[ \spr{M(\cdot)x}{x'} = K(\spr{N(\cdot)x}{x'})\quad (x\in D_N,x'\in X').\]
Then $(M(t) :\: t \in \Omega_2)$ is $R[E_2]$-bounded and
\[R[E_2](M(t):\:t\in \Omega_2) \leq \|K\| R[E_1](N(t):\:t \in \Omega_1).\]
\end{lem}

\begin{proof}
Let $x \in D_N$ and $x' \in X'.$
By \eqref{Equ R[E]-bounded general E} in Proposition \ref{Prop General E}, we have $\spr{N(\cdot)x}{x'} \in E_1',$
and thus, $\spr{M(\cdot)x}{x'} \in E_2'.$
For any $f \in E_2,$
\[\int_{\Omega_2} \spr{M(t)x}{x'} f(t) d\mu_2(t) = \int_{\Omega_1} \spr{N(t)x}{x'} (K'f)(t) d\mu_1(t) = \spr{N_{K'f} x}{x'}.\]
By assumption, the operator $N_{K'f}$ belongs to $B(X),$ and therefore also $M_f$ belongs to $B(X).$
Furthermore,
\begin{align}
R[E_2](M(t):\:t\in\Omega_2) & = R(\{M_f :\: \|f\|_{E_2} \leq 1\})
\nonumber \\
& = R(\{N_{K'f} :\: \|f\|_{E_2} \leq 1\})
\nonumber \\
& \leq \|K'\| R(\{N_{K'f} :\: \|K'f\|_{E_1} \leq 1\})
\nonumber\\
& \leq \|K\| R(\{N_g:\: \|g\|_{E_1} \leq 1\})
\nonumber \\
& = \|K\| R[E_1](N(t):\:t\in\Omega_1).
\nonumber
\end{align}
\end{proof}

In the following lemma, we collect some further simple manipulations of $R[E]$-boundedness.
Its proof is immediate from Definition \ref{Def R[E]-bounded}.

\begin{lem}\label{Lem R[E]-bounded technical}
Let $(\Omega,\mu)$ be a $\sigma$-finite measure space, let $E$ be as in \eqref{Equ Function space E}
and let $(N(t):\:t\in\Omega)$ satisfy (1) and (2) of Definition \ref{Def R[E]-bounded}.
\begin{enumerate}
\item Let $f \in L^\infty(\Omega)$ and $(N(t):\:t \in \Omega)$ be $R[L^p(\Omega)]$-bounded for some $1 \leq p \leq \infty.$
Then
\[R[L^p(\Omega)](f(t)N(t):\:t\in\Omega) \leq \|f\|_\infty R[L^p(\Omega)](N(t):\:t \in \Omega).\]
In particular, $R[L^p(\Omega_1)](N(t):\:t\in\Omega_1) \leq R[L^p(\Omega)](N(t):\:t\in\Omega)$ for any measurable subset $\Omega_1 \subset \Omega.$
\item\label{it Prop R[E]-bounded technical Subst}
 Let $w: \Omega \to (0,\infty)$ be measurable.
Then for $1 \leq p \leq \infty$ and $p'$ the conjugate exponent,
\[ R[L^p(\Omega,w(t) d\mu(t))](N(t):\:t\in\Omega) = R[L^p(\Omega,d\mu)](w(t)^{\frac{1}{p'}}N(t):\:t\in\Omega).\]
\item For $n \in \N,$ let $\varphi_n : \Omega \to \R_+$ with $\sum_{n=1}^\infty \varphi_n(t) = 1$ for all $t \in \Omega.$
Then
\[ R[E](N(t):\: t \in \Omega) \leq \sum_{n=1}^\infty R[E](\varphi_n(t)N(t):\: t \in \Omega).\]
\end{enumerate}
\end{lem}

We turn to applications to the functional calculus.
That is, the $R$-bounded functional calculus yields $R[L^2]$-bounded sets by the following proposition.
Here we may and do always choose the dense subset $D_N = D_A,$ the calculus core from \eqref{Equ D}.

\begin{defi}\label{Def R-bdd Matr R-bdd calculus}
Let $A$ be a $0$-sectorial operator.
Let $E \in \{\Ha,\Sea\}$.
We say that $A$ has an $R$-bounded $E$ calculus if $A$ has an $E$ calculus, which is an $R$-bounded mapping in the sense of \cite[Definition 2.7]{KrLM}, i.e.
\[ R(\{ f(A):\: \|f\|_E \leq 1\}) < \infty.\]
\end{defi}

In the next proposition we need the Mellin transform
\begin{equation}\label{equ Mellin transform}
M : L^2(\R_+,ds/s) \to L^2(\R,dt), f \mapsto (t \mapsto \int_0^\infty s^{it} f(s) ds/s) 
\end{equation}
which is an isometry.

\begin{prop}\label{Prop General Averaged}
Let $A$ be a $0$-sectorial operator having an $R$-bounded $\Sobexp^\alpha_2$ calculus for some $\alpha > \frac12.$
Let $\phi \in \Sobolev^\alpha_{2,\loc}(\R_+)$ such that $t \mapsto M \phi (t) \langle t \rangle^\alpha$ belongs to $L^\infty(\R),$
where $M$ denotes the Mellin transform.
Then $(\phi(tA) :\: t > 0)$ is $R[L^2(\R_+,\frac{dt}{t})]$-bounded with bound $\leq C \| M \phi(t) \langle t \rangle^\alpha \|_\infty.$
\end{prop}

\begin{proof}
Since $\phi_e(t+\log(s)) = \phi(se^{t}),$ we have to show that $(\phi_e(t + \log(A)) : \: t \in \R)$ is $R[L^2(\R)]$-bounded with the stated bound.
Then the proposition follows using the isometry $L^2(\R_+,\frac{dt}{t}) \to L^2(\R),\: f \mapsto f(e^{(\cdot)})$ and Lemma \ref{Lem Transformation R[E]-bounded}.
For $h \in L^2(\R)\cap L^1(\R)$ with, say, compact support, we have
\begin{equation}\label{Equ Convolution}
\int_\R h(-t) \phi_e(t + \log(A))x dt = (h \ast \phi_e)\circ \log (A) x \quad (x \in D_A).
\end{equation}
Indeed, for fixed $x \in D_A,$ there exists $\psi_0 \in C^\infty_c(\R)$ such that $\psi_0 \circ \log(A)x = x.$
Choose some $\psi \in C^\infty_c(\R)$ such that $\psi(r) = 1$ for $r \in \supp \psi_0 - \supp h,$ so that $\psi(t+\log(A))x = \psi(t+\log(A))\psi_0\circ \log(A)x = \psi_0\circ\log(A)x = x$ for any $-t \in \supp h.$
Then for any $x' \in X',$
\begin{align*}
\int_\R h(-t) \spr{\phi_e(t + \log(A))x}{x'} dt & = \int_\R h(-t) \spr{(\phi_e \psi)(t+\log(A))x}{x'} dt \\
& = \int_\R h(-t) \frac{1}{2 \pi} \int_\R (\phi_e \psi)\hat{\phantom{i}}(s) e^{ist} \spr{A^{is} x}{x'} ds dt \\
& = \frac{1}{2 \pi} \int_\R \left( \int_\R h(-t) e^{ist} (\phi_e \psi)\hat{\phantom{i}}(s) dt \right) \spr{A^{is} x}{x'} ds \\
& = \frac{1}{2 \pi} \int_\R \hat{h}(s) (\phi_e \psi)\hat{\phantom{i}}(s) \spr{A^{is} x}{x'} ds \\
& = \spr{(h \ast (\phi_e\psi))\circ\log(A) x}{x'}.
\end{align*}
where we used $h \in L^1(\R),\: \phi_e \psi \in \Sobolev^\alpha_2$ and $s \mapsto \langle s \rangle^{-\alpha} \spr{A^{is} x}{x'} \in L^2(\R),$ to apply Fubini in the third line, and Lemma \ref{Lem Sobolev calculus} in the last step.
We also have $(h\ast (\phi_e \psi))\psi_0 = (h \ast \phi_e)\psi_0$ and \eqref{Equ Convolution} follows.
Then the claim follows from $\|\phi_e \ast h \|_{\Sobolev^\alpha_2} \leq \| \hat{\phi_e}(t) \langle t \rangle^\alpha \|_{L^\infty(\R)} \|h\|_{L^2(\R)}$ and density of the above $h$ in $L^2(\R).$
\end{proof}

\begin{exa}
Consider $\phi(t) = t^{\beta} (e^{i \theta} - t)^{-1},$ where $\beta \in (0,1)$ and $|\theta| < \pi$ and $A$ an operator as in the proposition above.
Then $\phi(tA) = t^\beta A^\beta (e^{i\theta} - tA)^{-1} = t^{\beta - 1} A^{\beta} (e^{i\theta} t^{-1} - A)^{-1}$ is an $R[L^2(\frac{dt}{t})]$-bounded family with bound $\lesssim \theta^{-\alpha}.$
Indeed, $M \phi(t) = e^{i (\theta-\pi)(it + \beta - 1)} \frac{\pi}{\sin \pi (it + \beta - 1)}.$
As $|\sin\pi( it + \beta - 1 )| \cong \cosh(\pi t)$ for fixed $\beta,$ we have
$|M\phi(t) \langle t \rangle^\alpha| \lesssim e^{-(\theta - \pi)|t|} \langle t \rangle^\alpha \frac{1}{\cosh(\pi t)} \lesssim \theta^{-\alpha}.$
\end{exa}

Theorem \ref{Thm Characterizations Hoermander calculus} will show that a converse to Proposition \ref{Prop General Averaged} holds,
for many classical operator families including the above example, i.e. one can recover the $R$-bounded $\Sea$ calculus from averaged $R$-boundedness conditions.

\section{Main Results}\label{Sec Main Results}

We introduced the notion of $R[E]$-boundedness to give the following characterization of ($R$-bounded) $\Sea$ calculus.

\begin{thm}\label{Thm Characterizations Hoermander calculus}
Let $A$ be a $0$-sectorial operator on a Banach space $X$ with a bounded $\HI(\Sigma_\omega)$ calculus for some $\omega \in (0,\pi).$
Let $\alpha > \frac12.$
Consider the following conditions.\\

\noindent
\textit{Sobolev Calculus}
\begin{enumerate}
\item\label{it Wa c}$A$ has an $R$-bounded $\Sobexp^\alpha_2$ calculus.
\end{enumerate}
\textit{Imaginary powers}
\begin{enumerate}
\setcounter{enumi}{1}
\item\label{it group}$( \tma A^{it}:\:t \in \R)$ is $R[L^2(\R)]$-bounded.
\end{enumerate}
\textit{Resolvents}
\begin{enumerate}
\setcounter{enumi}{2}
\item\label{it Res A growth}For some/all $\beta \in (0,1)$ there exists $C>0$ such that for all $\theta \in (-\pi,\pi)\backslash\{0\}:\:
R[L^2(\R_+,dt/t)](t^\beta A^{1 - \beta}R(e^{i\theta}t,A):\:t > 0) \leq C |\theta|^{-\alpha}.$
\item\label{it Res A growth variant}For some/all $\beta \in (0,1)$ and $\theta_0 \in (0,\pi],$
$ ( |\theta|^{\alpha-\frac12} t^\beta A^{1 - \beta} R(e^{i\theta}t,A):\: 0 < |\theta| \leq \theta_0,\, t > 0 )$
is $R[L^2((0,\infty) \times [-\theta_0,\theta_0]\backslash \{0\},dt/t d\theta)]$-bounded.
\end{enumerate}
\textit{Analytic Semigroup} ($T(z) = e^{-zA}$)
\begin{enumerate}
\setcounter{enumi}{4}
\item\label{it Sgr growth}There exists $C>0$ such that for all $\theta \in (-\frac{\pi}{2},\frac{\pi}{2}):\:
R[L^2(\R_+)](A^{1/2}T(e^{i\theta}t):\:t>0) \leq C (\frac{\pi}{2}-|\theta|)^{-\alpha}.$
\item\label{it Sgr growth variant}$( \langle \frac{x}{y} \rangle^\alpha |x|^{-\frac12} A^{1/2} T(x+iy):\:x > 0,\,y \in \R)$ is  $R[L^2(\R_+\times \R)]$-bounded.
\end{enumerate}
\textit{Wave Operators}
\begin{enumerate}
\setcounter{enumi}{6}
\item\label{it Wave} $A$ has the auxiliary functional calculus $\Phi_A : \Sobexp^\gamma_2 \to B(D(\theta),X)$ from Section \ref{Sec Hormander classes} for some (possibly large) $\gamma > 0$ and $\theta > 0$ so that in particular, the operators $A^{-\alpha + \frac12} (e^{isA}-1)^m$ are densely defined for some $m > \alpha - \frac12.$
Assume moreover that $(|s|^{-\alpha}A^{-\alpha + \frac12}(e^{isA}-1)^m:\:s\in\R)$ is $R[L^2(\R)]$-bounded.
\item\label{it Wave variant}  $A$ has the auxiliary functional calculus $\Phi_A : \Sobexp^\gamma_2 \to B(D(\theta),X)$ from Section \ref{Sec Hormander classes} for some (possibly large) $\gamma > 0$ and $\theta > 0$ so that in particular, the operators $A^{\frac12 - \alpha} \left(e^{isA} - \sum_{j=0}^{m-1} \frac{(isA)^j}{j!} \right)$ are densely defined.
Assume moreover that
\[\left(A^{\frac12-\alpha} |s|^{-\alpha}\left(e^{isA} - \sum_{j=0}^{m-1} \frac{(isA)^j}{j!} \right):\:s \in \R\right)\]
 is $R[L^2(\R)]$-bounded.
\end{enumerate}
Then the following conditions are equivalent:
\[ \eqref{it Wa c}, \eqref{it group}, \eqref{it Res A growth variant}, \eqref{it Sgr growth variant}, \eqref{it Wave}.\]
The condition \eqref{it Wave variant} is also equivalent under the assumption that $\alpha - \frac12 \not\in \N_0$ and $m \in \N_0$ such that $\alpha - \frac12 \in (m,m+1).$ 

All these conditions imply the remaining ones \eqref{it Res A growth} and \eqref{it Sgr growth}.
If $X$ has property $(\alpha)$ then, conversely, these two conditions imply that $A$ has an $R$-bounded $\Sobexp^{\alpha + \epsilon}_2$ calculus for any $\epsilon > 0.$
\end{thm}

As a preparatory lemma for the proof of Theorem \ref{Thm Characterizations Hoermander calculus}, we state

\begin{lem}\label{Lem Wave Technical 3}
Let $\beta \in \R$ and $f(t) = f_m(\beta + it)$ with $f_m$ as in \eqref{Equ f_m}.
Then there exist $C,\epsilon,\delta > 0$ such that for any interval $I \subset \R$ with $|I| \geq C$ there is a subinterval $J \subset I$
with $|J| \geq \delta$ so that $|f(t)| \geq \epsilon$ for $t \in J.$
Consequently, for $N > C/\delta,$
\[ \sum_{k=-N}^N |f(t+k\delta)| \gtrsim 1 \quad (t \in J).\]
\end{lem}

\begin{proof}
Suppose for a moment that
\begin{equation}\label{Equ Lem Wave Technical 3}
\exists\,C,\,\epsilon > 0\,\forall\,I\text{ interval with }|I|\geq C\,\exists\,t \in I:\: |f(t)| \geq \epsilon.
\end{equation}
It is easy to see that $\sup_{t\in\R}|f'(t)| < \infty,$ so that for such a $t$ and $|s-t|\leq \delta = \delta(\|f'\|_\infty,\epsilon),$
$|f(s)| \geq \epsilon/2.$
Thus the lemma follows from \eqref{Equ Lem Wave Technical 3} with $J = B(t,\delta/2).$

It remains to show \eqref{Equ Lem Wave Technical 3}.
Suppose that this is false.
Then
\begin{equation}\label{Equ 2 Lem Wave Technical 3}
\forall\,C,\,\epsilon >0\,\exists\,I\text{ interval with }|I| \geq C:\:\forall\,t \in I:\: |f(t)| < \epsilon.
\end{equation}
Since $f''$ is bounded and $\|f'\|_{L^\infty(I)} \leq \sqrt{4 \|f\|_{L^\infty(I)} \max(\|f''\|_{L^\infty(I)},1)},$ see \cite[p.115 Exercice 15]{Rud},
we deduce that \eqref{Equ 2 Lem Wave Technical 3} holds for $f'$ in place of $f,$
and successively also for $f^{(n)}$ for any $n.$
But there is some $n\in\N$ such that $\inf_{t \in \R}|f^{(n)}(t)| > 0.$
Indeed,
\[ f^{(n)}(t) = \sum_{k=1}^m \alpha_k (-i \log k)^n e^{-it\log k},\]
with $\alpha_k = \binom{m}{k} (-1)^{m-k} k^{-\beta} \neq 0,$ whence
\[|f^{(n)}(t)| \geq |\alpha_m| \, |\log m|^n - \sum_{k=1}^{m-1}|\alpha_k|\,|\log k|^n > 0\]
for $n$ large enough.
This contradicts \eqref{Equ 2 Lem Wave Technical 3}, so that the lemma is proved.
\end{proof}

\begin{proof}[Proof of Theorem \ref{Thm Characterizations Hoermander calculus}]
~\\
\noindent
\eqref{it Wa c} $\Leftrightarrow$ \eqref{it group}:
By the $\Sea$ representation formula \eqref{Equ Soa calculus formula}, we have
$R[L^2(\R,dt)]( \tma A^{it}: \: t \in \R) = R\left(\left\{ \int_\R f(t) \tma A^{it} dt :\: \|f\|_{L^2(\R)} \leq 1 \right\} \right) = R\left(\left\{ 2 \pi f(A) :\: \|f\|_{\Sea} \leq  1 \right\} \right).$

The strategy to show the stated remaining (almost) equivalences between \eqref{it group} and \eqref{it Res A growth} -- \eqref{it Wave} consists more or less in finding an integral transform $K$ as in Lemma \ref{Lem Transformation R[E]-bounded} mapping the imaginary powers $A^{it}$ to resolvents, to the analytic semigroup and to the wave operators, and vice versa.

\noindent
\eqref{it group} $\Rightarrow$ \eqref{it Wave}:
By the above shown equivalence of \eqref{it Wa c} and \eqref{it group}, $A$ has an $R$-bounded $\Sea$ calculus, thus in particular an auxiliary calculus $\Phi_A : \Sobexp^\gamma_2 \to B(D(\theta),X)$ with $\gamma = \alpha$ and $\theta = 0.$
By Proposition \ref{Prop General E}, we clearly have that $\| t \mapsto \tma \spr{A^{it} x}{x'} \|_{L^2(\R,dt)} \leq C \| x \| \, \| x' \|,$ provided \eqref{it group} holds.
Thus, by Proposition \ref{Prop Wave BIP Function}, and Lemmas \ref{Lem Transformation R[E]-bounded} and \ref{Lem R[E]-bounded technical} (1), using the fact that the Mellin transform from \eqref{equ Mellin transform} is an isometry, \eqref{it Wave} follows.

\noindent
\eqref{it Wave} $\Rightarrow$ \eqref{it group}:
Recall the function $h_\mp$ from Proposition \ref{Prop Wave BIP Function}.
By the Euler Gamma function development \cite[p.~15]{Lebe}, we have the lower estimate
\[ |h_\mp(t)| \gtrsim |f_m(\frac12 - \alpha + it)| e^{\frac{\pi}{2}(\pm t - |t|)} \tma. \]
Thus, by Proposition \ref{Prop Wave BIP Function}, and Lemmas \ref{Lem Transformation R[E]-bounded} and \ref{Lem R[E]-bounded technical} (1),
\begin{equation}\label{Equ Proof Wave Op Converse}
 \left( \tma f_m(\frac12 - \alpha + it) A^{-it} : \: t \in \R \right) \text{ is }R[L^2(\R)]\text{-bounded.}
\end{equation}
To get rid of $f_m$ in this expression, we apply Lemma \ref{Lem Wave Technical 3}.
According to that lemma, we have $N \in \N$ and $\delta > 0$ such that
$\sum_{k = -N}^N |f(t + k \delta)| \gtrsim 1$ for any $t \in \R$ and $f(t) = f_m(\frac12 - \alpha + it).$
Write
\[ \sum_{k=-N}^N f(t+ k \delta) \tma  A^{-it} = \sum_{k=-N}^N \left[\frac{\langle t + k \delta \rangle^\alpha}{\langle t \rangle^\alpha} A^{ik\delta}\right]\left[ f(t+k \delta) \langle t + k \delta\rangle^{-\alpha} A^{-i(t+k\delta)}\right].\]
By \eqref{Equ Proof Wave Op Converse}, the term in the second brackets is $R[L^2(\R)]$-bounded.
The term in the first brackets is a bounded function times a bounded operator, due to the assumption that $A$ has a bounded $H^\infty(\Sigma_\omega)$ calculus.
Thus, the right hand side is $R[L^2(\R)]$-bounded, and so the left hand side is.
Now appeal once again to Lemma \ref{Lem R[E]-bounded technical} (1) to deduce \eqref{it group}.\\

\noindent
\eqref{it group} $\Rightarrow$ \eqref{it Res A growth}:
We fix $\theta \in (-\pi,\pi)$ and set
\[ K_\theta:  L^2(\R,ds) \to L^2(\R,ds),\, f(s) \mapsto (\pi - |\theta|)^{\alpha} \frac{1}{\sin\pi (\beta + is)}  e^{\theta s} \langle s \rangle^{\alpha} f(s). \]
We have
\[\sup_{|\theta| < \pi} \|K_\theta\| = \sup_{|\theta| < \pi,\,s\in\R} \langle s \rangle^\alpha (\pi-|\theta|)^\alpha
\frac{e^{\theta s}}{|\sin \pi (\beta + is)|} \lesssim \sup_{\theta,s} \langle s(\pi-|\theta|) \rangle^\alpha e^{-|s|(\pi-|\theta|)} < \infty.\]
In \cite[p. 228 and Theorem 15.18]{KuWe}, the following formula is derived for $x \in A(D(A^2))$ and $|\theta| < \pi:$
\begin{equation}\label{Equ Res A Mellin BIP} \frac{\pi}{\sin \pi(\beta + is)} e^{\theta s} A^{is} x = \int_0^\infty t^{is} \left[ t^\beta e^{i\theta \beta} A^{1- \beta} (e^{i\theta} t + A)^{-1} x \right] \frac{dt}{t}.
\end{equation}
Thus, with $R(\lambda,A) = (\lambda - A)^{-1},$
\begin{align}
\sup_{0 < |\theta| \leq \pi} |\theta|^{\alpha} R[L^2(\R_+,dt/t)](t^\beta A^{1-\beta}R(te^{i\theta},A))
& = \sup_{|\theta| < \pi} (\pi-|\theta|)^{\alpha} R[L^2(\R_+,dt/t)](t^{\beta} A^{1 - \beta} (e^{i\theta} t +A)^{-1})
\nonumber \\
& = \sup_{|\theta| < \pi} (\pi-|\theta|)^{\alpha} R[L^2(\R,ds)](\frac{\pi}{\sin\pi (\beta + is)}e^{\theta s}A^{is})
\label{Equ Proof Thm Charact} \\
& \lesssim R[L^2(\R,ds)] (\langle s \rangle^{-\alpha} A^{is}).
\nonumber
\end{align}

\noindent
Next we claim that for any $\epsilon > 0,$
\eqref{it Res A growth} implies \eqref{it group}, where in \eqref{it group}, $\alpha$ is replaced by $\alpha + \epsilon.$\\
First we consider $\langle s \rangle^{-(\alpha+\epsilon)}A^{is}x$ for $s\geq 1.$
By Lemma \ref{Lem R[E]-bounded technical} (3),
\begin{align}
R[L^2([1,\infty),ds)](\langle s \rangle^{-(\alpha+\epsilon)}A^{is})
& \leq \sum_{n = 0}^\infty R[L^2([2^n,2^{n+1}])](\langle s \rangle^{-\epsilon} \langle s \rangle^{-\alpha}A^{is}) \nonumber \\
& \leq \sum_{n = 0}^\infty 2^{-n\epsilon} R[L^2([2^n,2^{n+1}])] (\langle s \rangle^{-\alpha} A^{is}).
\label{Equ Proof Thm Charact 2}
\end{align}
For $s \in [2^n,2^{n+1}],$ we have
\[\langle s \rangle^{-\alpha} \lesssim 2^{-n\alpha} \lesssim 2^{-n\alpha} e^{-2^{-n}s} \lesssim (\pi-\theta_n)^\alpha \frac{e^{\theta_n s}}{\sin \pi(\beta + i s)} ,\]
where $\theta_n = \pi - 2^{-n}.$
Therefore
\begin{align}
R[L^2([2^n,2^{n+1}])] ( \langle s \rangle^{-\alpha} A^{is})
& \lesssim (\pi-\theta_n)^{\alpha} R[L^2(\R,ds)] ( \frac{\pi}{\sin\pi (\beta + is)}e^{\theta_n s}A^{is}) \nonumber \\
& \overset{\eqref{Equ Proof Thm Charact}}{\lesssim} \sup_{0<|\theta| \leq \pi} |\theta|^{\alpha} R[L^2(\R_+,dt/t)] (t^\beta A^{1 - \beta}R(te^{i\theta},A)) < \infty.
\nonumber
\end{align}
Thus, the sum in \eqref{Equ Proof Thm Charact 2} is finite.

The part $\langle s \rangle^{-(\alpha + \epsilon)}A^{is}$ for $s \leq -1$ is treated similarly,
whereas $R[L^2(-1,1)](\langle s \rangle^{-\alpha}A^{is}) \cong R[L^2(-1,1)](A^{is}).$
It remains to show that the last expression is finite.
We have assumed that $X$ has property $(\alpha).$
Then the fact that $A$ has an $\HI$ calculus implies that $\{A^{is}:\: |s| < 1\}$ is $R$-bounded \cite[Theorem 12.8]{KuWe}, so by Example \ref{Exa R[E]-bounded} a), it is $R[L^1(-1,1)]$-bounded.
For $f \in L^2(-1,1),$ we have $\|f\|_1 \leq C \|f\|_2,$ and consequently,
\[ \left\{ \int_{-1}^1 f(s) A^{is} ds :\: \|f\|_2 \leq 1 \right\} \subset C \left\{ \int_{-1}^1 f(s) A^{is} ds :\: \|f\|_1 \leq 1 \right\}. \]
In other words, $(A^{is}: |s| < 1)$ is $R[L^2]$-bounded.\\

\noindent
\eqref{it group} $\Longleftrightarrow$ \eqref{it Res A growth variant}:\\
Consider
\begin{equation}\label{Equ Proof Charact K}
 K : L^2(\R,ds) \to L^2(\R \times (-\pi,\pi),ds d\theta),\,
f(s) \mapsto (\pi-|\theta|)^{\alpha-\frac12}\frac{1}{\sin \pi(\beta + is)} e^{\theta s} \langle s \rangle^{\alpha} f(s),
\end{equation}
Note that $|\sin \pi(\beta + is)| \cong \cosh(\pi s)$ for $\beta \in (0,1)$ fixed.
$K$ is an isomorphic embedding.
Indeed,
\[ \|Kf\|_2^2 =
\int_\R \int_{-\pi}^\pi \left((\pi-|\theta|)^{\alpha-\frac12} e^{\theta s}\right)^2 d\theta \frac{1}{|\sin^2(\pi(\beta + i s))|} \langle s \rangle^{2\alpha} |f(s)|^2 ds\]
and
\begin{align}
\int_{-\pi}^\pi (\pi-|\theta|)^{2\alpha -1} e^{2\theta s} d\theta
& \cong \int_0^\pi \theta^{2\alpha -1} e^{2(\pi - \theta) |s|} d\theta
\nonumber \\
& \cong \cosh^2(\pi s) \int_0^\pi \theta^{2\alpha-1} e^{-2\theta |s|} d\theta.
\nonumber
\end{align}
For $|s|\geq 1,$
\[\int_0^\pi \theta^{2\alpha -1} e^{-2\theta |s|} d\theta = (2|s|)^{-2\alpha} \int_0^{2|s|\pi} \theta^{2\alpha -1} e^{-\theta}d\theta \cong |s|^{-2\alpha}.\]
This clearly implies that $\|Kf\|_2 \cong \|f\|_2.$
We now apply Lemma \ref{Lem Transformation R[E]-bounded} for $K$ and for the mapping $L : \Im K \oplus (\Im K)^\bot \to L^2(\R),\: x \oplus y \mapsto K^{-1} x.$
Note that $\Im K$ is closed since $K$ is an isomorphic embedding, so that $\Im K \oplus (\Im K)^\bot = L^2(\R \times (-\pi,\pi))$ and $L$ is bounded since $\|L(x \oplus y)\| = \|K^{-1} x\| \cong \|x\| \leq \|x \oplus y\|.$
We deduce
\[ R[L^2(\R,ds)](\langle s \rangle^{-\alpha} A^{is})
\cong R[L^2(\R \times (-\pi,\pi),dsd\theta)]((\pi-|\theta|)^{\alpha-\frac12} \frac{1}{\cosh (\pi s)} e^{\theta s}A^{is}).\]
Recall the formula \eqref{Equ Res A Mellin BIP}, i.e.
\[\frac{\pi}{\sin \pi (\beta + i s)} e^{\theta s} A^{is} x =
\int_0^\infty t^{is}\left[t^{\beta} e^{i\theta \beta} A^{1 - \beta}(e^{i\theta}t+A)^{-1}x\right] \frac{dt}{t}\]
for $|\theta| < \pi$ and $x \in A(D(A^2)).$
Note that $A(D(A^2))$ is a dense subset of $X.$
As the Mellin transform $f(s) \mapsto \int_0^\infty t^{is} f(s) \frac{ds}{s}$ is an isometry $L^2(\R_+,\frac{ds}s)\to L^2(\R,dt),$ we get by Lemma \ref{Lem Transformation R[E]-bounded}
\begin{align}
R[L^2(\R)] (\langle s \rangle^{-\alpha} A^{is})
& \cong R[L^2(\R_+ \times (-\pi,\pi),\frac{dt}{t}d\theta)]((\pi -|\theta|)^{\alpha-\frac12} t^{\beta}A^{1-\beta}(e^{i\theta}t+A)^{-1})
\nonumber \\
& \cong R[L^2(\R_+ \times (0,2\pi),dt/t d\theta)] ( |\theta|^{\alpha-\frac12} t^\beta A^{1 - \beta} R(e^{i\theta}t,A) ).
\nonumber
\end{align}
so that \eqref{it group} $\Longleftrightarrow$ \eqref{it Res A growth variant} for $\theta_0 = \pi.$
Here we used that $(e^{i\theta} t + A)^{-1} = - (e^{i (\pm \pi + \theta)} t - A)^{-1} = - R(e^{i(\pm \pi + \theta)} t,A).$

For a general $\theta_0 \in (0,\pi],$ consider $K$ from \eqref{Equ Proof Charact K} with restricted image, i.e.
\[K : L^2(\R,ds) \to L^2(\R \times (-\pi,-(\pi-\theta_0)] \cup [\pi-\theta_0,\pi),ds d\theta).\]
Then argue as in the case $\theta_0 = \pi.$\\

\noindent
\eqref{it Res A growth variant} $\Longleftrightarrow$ \eqref{it Sgr growth variant}:\\
The proof of \eqref{it group} $\Longleftrightarrow$ \eqref{it Res A growth variant} above shows that condition \eqref{it Res A growth variant} is independent of $\theta_0 \in (0,\pi]$ and $\beta \in (0,1).$
Put $\theta_0 = \pi$ and $\beta = \frac12.$
Apply Lemma \ref{Lem Transformation R[E]-bounded} with
\[ (e^{i\theta} \mu + it)^{-1} = K[\exp(-(\cdot)e^{i\theta} \mu)\chi_{(0,\infty)}(\cdot)](t), \]
where $K : L^2(\R,ds) \to L^2(\R,dt)$ is the Fourier transform.
This yields that \eqref{it Res A growth variant} is equivalent to
\[ R[L^2((0,\frac{\pi}{2}) \times \R_+,d\theta dt)] (|\theta|^{\alpha-\frac12} A^{\frac12} T(\exp(\pm i(\frac{\pi}{2} - \theta))t)) < \infty.\]
Applying the change of variables $\theta \rightsquigarrow \frac{\pi}{2} \pm \theta$ and $dt \rightsquigarrow t dt$ shows that this is equivalent to
\[ R[L^2((-\frac{\pi}2,\frac{\pi}2)\times \R_+,d\theta t dt)] ( (\frac{\pi}{2} - |\theta|)^{\alpha - \frac12} t^{-\frac12}A^{\frac12} T(e^{i\theta}t) ) < \infty.\]
Now the equivalence to  \eqref{it Sgr growth variant} follows from the change of variables $a = t \cos \theta,\,b= t \sin \theta,\,t = |a + ib|,\,d\theta\, t dt = da\, db.$\\

\noindent
\eqref{it Res A growth} $\Longleftrightarrow$ \eqref{it Sgr growth} for $\beta = \frac12:$
Use $K$ and the first argument from the proof of \eqref{it Res A growth variant} $\Longleftrightarrow$ \eqref{it Sgr growth variant}.\\

\noindent
\eqref{it group} $\Longleftrightarrow$ \eqref{it Wave variant}:
Recall the formula from Proposition \ref{Prop Wave variant}, $M$ denoting the Mellin transform,
\[ M(\langle (sA)^{\frac12} w_\alpha(sA) x, x' \rangle)(t) = i^{-\alpha + \frac12 + it} \Gamma(-\alpha + \frac12 + it) \langle A^{-it} x, x' \rangle,\]
where
$w_\alpha(s) = |s|^{-\alpha} \left( e^{is} - \sum_{j=0}^{m-1} \frac{(is)^j}{j!} \right).$
Then we have according to \cite[p.~15]{Lebe}, since $\alpha-\frac12 \not\in \N_0,$
\[|i^{-\alpha + \frac12 + it} \cdot \Gamma(-\alpha + \frac12 + it)| \cong e^{-\frac{\pi}{2} t} \cdot e^{-\frac{\pi}{2} |t|} \tma\]
for $t \in \R.$
Thus, with Lemmas \ref{Lem Transformation R[E]-bounded} and \ref{Lem R[E]-bounded technical} (2),
\begin{align*}
R[L^2(\R,dt)](\tma A^{it}) < \infty & \Longleftrightarrow R[L^2(\R_+,ds/s)]((sA)^{\frac12} w_\alpha(\pm sA)) < \infty \\
& \Longleftrightarrow R[L^2(\R,ds)](A^{\frac12} w_\alpha(sA)) < \infty.
\end{align*}
\end{proof}

We can now complete the proof of a claim in Section \ref{Sec Hormander classes}.

\begin{prop}\label{Prop assumptions auxiliary calculus}
Let $A$ be a $0$-sectorial operator and $\beta, \theta > 0.$
Assume one of the following conditions.
\begin{align}
 \int_\R |\langle s \rangle^{-\beta} \langle e^{isA} x, x' \rangle|^2 ds & \leq C \|x\|_{D(\theta)}^2 \|x'\|_{X'}^2, \label{equ 2a assumptions auxiliary calculus} \\
\int_0^\infty | \langle \exp(-e^{i \omega} t A)x, x' \rangle|^2 dt & \leq C (\frac{\pi}{2} - |\omega|)^{-2\beta} \|x\|_{D(\theta)}^2 \|x'\|_{X'}^2\text{ for some fixed }C>0\text{ and any }\frac{\pi}{4} \leq |\omega| < \frac{\pi}{2}, \label{equ 3a assumptions auxiliary calculus} \\
 \int_0^\infty | t^\gamma \langle R(e^{i\omega} t, A) x, x' \rangle |^2 \frac{dt}{t} & \leq C |\omega|^{-2 \beta} \|x\|_{D(\theta)}^2 \|x'\|_{X'}^2\text{ for some fixed }\gamma \in (0,1)\text{ and any }\omega \in (-\pi,\pi) \backslash \{ 0 \}. \label{equ 4a assumptions auxiliary calculus}
\end{align}
Then $A$ satisfies
\begin{equation}\label{equ 1a assumptions auxiliary calculus}
\int_\R | \langle t \rangle^{-\alpha} \langle A^{it} x, x' \rangle |^2 dt \leq C \|x\|_{D(\theta')}^2 \|x'\|_{X'}^2
\end{equation}
for some $\alpha,\theta' > 0.$
\end{prop}

\begin{proof}
Assume that \eqref{equ 2a assumptions auxiliary calculus} holds.
We have with $e^{i\omega} t = r + is,$
\begin{align*}
\left( \frac{r}{|s|}\right)^\alpha \exp(-(r+is)A) & = \left[ \left( \frac{r}{|s|} \right)^\alpha ( 1 + rA )^{-\alpha} ( 1 + |s| )^\alpha (1 + A)^\alpha \right] \times \\
&\times \left[ ( 1 + |s| )^{-\alpha} (1 + A)^{-\alpha} e^{-isA} \right] \left[ ( 1 + rA)^\alpha \exp(-rA) \right]. 
\end{align*}
We handle each of the three brackets separately.
For the first bracket, note that the $\HI(\Sigma_\sigma)$ norm of $\lambda \mapsto \left( \frac{r}{|s|} \right)^\alpha ( 1 + r \lambda )^{-\alpha} ( 1 + |s| )^\alpha (1 + \lambda)^\alpha \rho^{\alpha}(\lambda)$ (where we recall $\rho(\lambda) = \lambda ( 1 + \lambda)^{-2}$) is uniformly bounded for $|s| \geq r > 0,$ thus the first bracket defines a bounded operator $D(\alpha + 1) \to X$ with uniform norm bound in $s,r.$
The third bracket is also uniformly bounded $X \to X.$
Thus, we deduce \eqref{equ 3a assumptions auxiliary calculus} with the same $\beta$ and $\theta$ replaced by $\theta + \alpha + 1.$

If \eqref{equ 3a assumptions auxiliary calculus} in turn holds, then by the boundedness of $A^{\frac12} : D(\frac12) \to X,$ we also have
\begin{equation}\label{equ Proof Prop auxiliary}
 \int_0^\infty | \langle A^{\frac12} \exp(-e^{i\omega} tA)x,x' \rangle|^2 dt \leq C (\frac{\pi}{2} - |\omega|)^{-2\beta} \|x\|_{D(\theta + \frac12)}^2 \|x'\|_{X'}^2
\end{equation}
for $|\omega| < \frac{\pi}{2}.$

By essentially the same proof as Theorem \ref{Thm Characterizations Hoermander calculus}, i.e. using the integral transforms from ``\eqref{it Res A growth variant} $\Longleftrightarrow$ \eqref{it Sgr growth variant}'' and ``\eqref{it Res A growth} $\Longrightarrow$ \eqref{it group}'' to go forth and back between the different operator families, and using that $\|A^{it}\|_{B(D(\sigma),X)} \leq C$ for $|t| \leq 1$ and any $\sigma > 0,$
one can show that this implies \eqref{equ 1a assumptions auxiliary calculus} with $\alpha > \beta$ and $\theta' = \theta.$

Finally, again arguing as in the proof of Theorem \ref{Thm Characterizations Hoermander calculus} and using that $A^{\frac12}$ is bounded $D(\frac{1}{2}) \to X,$ one can show that \eqref{equ 4a assumptions auxiliary calculus} implies \eqref{equ Proof Prop auxiliary}, which in turn implies \eqref{equ 1a assumptions auxiliary calculus}.
\end{proof}

Theorem \ref{Thm Characterizations Hoermander calculus} shows that averaged $R$-boundedness yields a good tool to describe $\Sea$ functional calculus.
However, many of the functions $f$ that correspond to relevant spectral multipliers, as for example in \eqref{it group} -- \eqref{it Wave} above, are not covered themselves by this calculus.
To pass from the $\Sea$ calculus to the $\Ha$ calculus, which does cover all the spectral multipliers alluded to above, we shall use the spectral decomposition of Paley-Littlewood type used in the proof of the following Theorem.

\begin{thm}\label{Thm Sea Ha}
Let $A$ be a $0$-sectorial operator on a Banach space $X$ with property $(\alpha)$ having a bounded $\HI(\Sigma_\sigma)$ calculus for some $\sigma \in (0,\pi].$
Then the following are equivalent for $\alpha > \frac12.$
\begin{enumerate}
\item $A$ has an $R$-bounded $\Sea$ calculus.
\item $A$ has an $R$-bounded $\Ha$ calculus.
\end{enumerate}
\end{thm}

\begin{exa}
Consider the operator $A = -\Delta$ on $X = L^p(\R^d)$ for some $1 < p < \infty$ and $d \in \N.$
H\"ormander's classical result states that $A$ has a bounded $\Hor^\alpha_2$ calculus for $\alpha > \frac{d}{2}.$
In fact, a stronger result holds and $A$ has an $R$-bounded $\Hor^\alpha_2$ calculus for the same range $\alpha > \frac{d}{2}.$
This is proved in \cite[Theorem 5.1]{Kr2}, \cite[Beginning of Section 4]{Kr3}.
\end{exa}

\begin{proof}[Proof of Theorem \ref{Thm Sea Ha}]
As $\Sea \subset \Ha,$ only the implication (1) $\Longrightarrow$ (2) has to be shown.
Consider a function $\phi \in \HI_0(\Sigma_\nu)$ such that $\sum_{n \in \Z}\phi^3(2^{-n} \lambda) = 1$ for any $\lambda \in \Sigma_\nu$ and some $\nu > \sigma.$
Furthermore, consider a function $\eta \in C^\infty_c$ with $\supp \eta \subset [\frac12,2]$ such that $\sum_{n \in \Z} \eta(2^{-n} t) = 1$ for any $t > 0.$
Let $f_1,\ldots,f_N \in C^\infty(\R_+)$ with $\|f_j\|_{\Ha} \leq 1$ for $j = 1,\ldots,N.$
Then for $x_1, \ldots, x_N$ belonging to the dense set $D_A \subseteq D(f_j(A))$ from \eqref{Equ D},
\begin{align*}
\E \| \sum_{j=1}^N \epsilon_j f_j(A) x_j \| & = \E \| \sum_{j=1}^N \sum_{k \in \Z} \epsilon_j \eta(2^{-k}A)f_j(A)x_j \| \\
& \cong \E \E' \| \sum_{j =1}^N \sum_{n,k \in \Z} \epsilon_j \epsilon_n' \phi^2(2^{-n}A) \eta(2^{-k}A)f_j(A) x_j \| \\
& \leq \sum_{l \in \Z} \E \E' \| \sum_{j=1}^N \sum_{n \in \Z} \epsilon_j \epsilon_n' [\phi(2^{-n}A) \eta(2^{-n-l}(A) f_j(A)] \phi(2^{-n} A) x_j\| \\
& \leq \left( \sum_{l \in \Z} C_l \right) \E \E' \| \sum_{j=1}^N \sum_{n \in \Z} \epsilon_j \epsilon_n' \phi(2^{-n} A)x_j \| \\
& \lesssim  \left( \sum_{l \in \Z} C_l \right) \E \|\sum_{j=1}^N \epsilon_j x_j \|,
\end{align*}
where we have used that $\|x\| \cong \E \| \sum_{n \in \Z} \epsilon_n \phi^2(2^{-n}A) x \| \cong \E \| \sum_{n \in \Z} \epsilon_n \phi(2^{-n}A) x \|.$
Indeed, the second expression is estimated by the third one, since $\{ \phi(2^{-n}A):\:n \in \Z \}$ is $R$-bounded by the $R$-boundedness of the $\HI(\Sigma_\nu)$ calculus \cite[12.8 Theorem]{KuWe}.
The third expression is estimated by the first one according to \cite[12.2 Theorem and 12.3 Remark]{KuWe}.
Finally the first expression is estimated by the second one again by \cite[12.2 Theorem and 12.3 Remark]{KuWe} and $|\langle x, x' \rangle| = | \E \langle \sum_{n \in \Z} \epsilon_n \phi^2(2^{-n}A)x, \sum_{k \in \Z} \epsilon_k \phi(2^{-k} A)x' \rangle| \leq \E \|\sum_n \epsilon_n \phi^2(2^{-n}A) x\| \: \E \| \sum_k \epsilon_k \phi(2^{-k}A)'x'\| \lesssim \E \|\sum_n \epsilon_n \phi^2(2^{-n}A) x\| \: \|x'\|. $

Furthermore, we used property $(\alpha)$ in the fourth line, and $C_l = R(\{ \phi(2^{-n}A)\eta(2^{-n-l}A)f_j(A) :\: n \in \Z,\: j = 1,\ldots,N \})$ and 
\begin{align*}
C_l & \lesssim \sup_{j=1,\ldots,N}\sup_{n \in \Z} \|\phi(2^{l}\cdot)\eta f_j(2^{n+l}\cdot)\|_{\Sea} \\
& \lesssim \sup_{j=1,\ldots,N}\sup_{k \in \Z} \|\eta f_j(2^k\cdot)\|_{\Sea} \sup_{m=0,\ldots,\lfloor \alpha \rfloor + 1} \sup_{t \in [\frac12,2]} t^m | \frac{d^m}{dt^m}\phi(2^l \cdot)(t)| \\
& \lesssim \sup_{j=1,\ldots,N}\|f_j\|_{\Ha}  2^{-\epsilon |l|}, \\
& \leq 2^{-\epsilon |l|}
\end{align*}
where $\epsilon > 0$ and we used the fact that $\phi \in \HI_0(\Sigma_\nu).$
Hence $\sum_{l\in\Z} C_l \lesssim \sup_{j=1,\ldots,N}\|f_j\|_{\Ha} < \infty.$
We have shown that
\begin{equation}\label{Equ Proof Localization Principle}
\{ f(A):\: f \in C^\infty(\R_+),\,\|f\|_{\Ha} \leq 1\}
\end{equation}
is $R$-bounded.
In particular, since $C^\infty(\R_+) \supset \bigcap_{\omega > 0}\HI(\Sigma_\omega),$
$A$ has a bounded $\Ha$ calculus in the sense of Definition \ref{Def Hor calculus}, and
by taking the closure of \eqref{Equ Proof Localization Principle}, this calculus is $R$-bounded.
\end{proof}

\section{Bisectorial operators and operators of strip type}\label{Sec Remarks}

\subsection{Bisectorial operators}

In this short subsection we indicate how to extend our results to bisectorial operators.
An operator $A$ with dense domain on a Banach space $X$ is called bisectorial of angle $\omega \in [0,\frac{\pi}{2})$ if it is closed, its spectrum is contained in the closure of $S_\omega = \{ z \in \C :\: |\arg(\pm z)| < \omega \},$ and one has the resolvent estimate 
\[\|(I+\lambda A)^{-1}\|_{B(X)} \leq C_{\omega'},\: \forall \: \lambda \not\in S_{\omega'},\: \omega' > \omega .\]
If $X$ is reflexive, then for such an operator we have again a decomposition $X = N(A) \oplus \overline{R(A)},$ so that we may assume that $A$ is injective.
The $\HI(S_\omega)$ calculus is defined as in \eqref{Equ Cauchy Integral Formula}, but now we integrate over the boundary of the double sector $S_\omega.$
If $A$ has a bounded $\HI(S_\omega)$ calculus, or more generally, if we have $\|Ax\| \cong \|(-A^2)^{\frac12}x\|$ for $x \in D(A) = D((-A^2)^{\frac12})$ (see e.g. \cite{DW}), then the spectral projections $P_1,\: P_2$ with respect to $\Sigma_1 = S_\omega \cap \C_+,\: \Sigma_2 = S_\omega \cap \C_-$ give a decomposition $X = X_1 \oplus X_2$ of $X$ into invariant subspaces for resolvents of $A$ such that the part $A_1$ of $A$ to $X_1$ and $-A_2$ of $-A$ to $X_2$ are sectorial operators with $\sigma(A_i) \subset \Sigma_i.$
For $f \in \HI_0(S_\omega)$ we have 
\begin{equation}\label{Equ Bisectorial}
f(A)x = f|_{\Sigma_1}(A_1)P_1 x + f|_{\Sigma_2}(A_2)P_2 x.
\end{equation}
We define the H\"ormander class $\Ha(\R)$ on $\R$ by $f \in \Ha(\R)$ if $f \chi_{\R_+} \in \Ha$ and $f(-\cdot) \chi_{\R_+} \in \Ha.$
Let $A$ be a $0$-bisectorial operator, i.e. $A$ is $\omega$-bisectorial for all $\omega > 0.$
Assume that $A$ has a bounded $\HI(S_\omega)$ calculus for some $\omega \in (0,\pi/2).$
Then $A$ has an ($R$-bounded) $\Ha(\R)$ calculus if the set $\{f(A):\: f \in \bigcap_{0 < \omega < \pi/2} \HI(S_\omega) \cap \Ha(\R),\: \|f\|_{\Ha(\R)}\leq 1\}$ is ($R$-)bounded.
Clearly, $A$ has an ($R$-bounded) $\Ha(\R)$ calculus if and only if $A_1$ and $-A_2$ have an ($R$-bounded) $\Ha$ calculus and in this case \eqref{Equ Bisectorial} holds again.

Let $f_t(\lambda) = \begin{cases} \lambda^{it} :\: \Re \lambda > 0 \\ (-\lambda)^{it} :\: \Re \lambda  < 0 \end{cases}.$
Then $f_t \in \HI(S_\omega)$ for any $\omega \in (0,\frac{\pi}{2}).$
Clearly, one has $f_t(A) = A_1^{it} \oplus (-A_2)^{it}$ on $X = X_1 \oplus X_2.$
It is easy to show that $(\tma f_t(A) :\: t \in \R)$ is $R[L^2(\R,dt)]$-bounded if and only if $(\tma A_1^{it} :\: t \in \R)$ and $(\tma (-A_2)^{it} :\: t \in \R)$ are both $R[L^2(\R,dt)]$-bounded.
Let $|A| = f(A)$ with $f(z) = 1$ for $\Re z > 0$ and $f(z) = -1$ for $\Re z < 0.$
Then similarly, we have that 
\[R[L^2(\R_+,dt/t)](t^\beta |A|^{1-\beta} (e^{i\theta}t - A)^{-1}) \lesssim (\min(|\theta|,\pi-|\theta|))^{-\alpha}\] for $0 < |\theta| < \pi$ if and only if both of the following conditions hold:
\[R[L^2(\R_+,dt/t)](t^\beta A_1^{1-\beta} (e^{i\theta} t - A_1)^{-1}) \lesssim |\theta|^{-\alpha}\text{ and }R[L^2(\R_+,dt/t)](t^\beta (-A_2)^{1-\beta} (e^{i\theta}t + A_2)^{-1} \lesssim |\theta|^{-\alpha}\] for $0 < |\theta| \leq \frac{\pi}{2}.$
Finally, we have that $|s|^{-\alpha} |A|^{-\alpha + \frac12} (e^{isA} - 1)^m$ is $R[L^2(\R)]$-bounded if and only if both $|s|^{-\alpha} A_1^{-\alpha + \frac12} (e^{isA_1} - 1)^m$ and
$|s|^{-\alpha} (-A_2)^{-\alpha + \frac12} (e^{isA_2} -1)^m)$ are $R[L^2(\R)]$-bounded.

Then using the projections $P_1$ and $P_2,$ it is clear how our main Theorems \ref{Thm Characterizations Hoermander calculus} and \ref{Thm Sea Ha} extend to bisectorial operators.

\subsection{Strip-type operators}

For $\omega > 0$ we let $\Str_\omega = \{z \in \C :\: |\Im z| < \omega\}$ the horizontal strip of height $2 \omega.$
We further define $\HI(\Str_\omega)$ to be the space of bounded holomorphic functions on $\Str_\omega,$ which is a Banach algebra equipped with the norm $\|f\|_{\infty,\omega} = \sup_{\lambda \in \Str_\omega} |f(\lambda)|.$
A densely defined operator $B$ is called $\omega$-strip-type operator if $\sigma(B) \subset \overline{\Str_\omega}$ and for all $\theta > \omega$ there is a $C_\theta > 0$ such that $\|\lambda (\lambda - B)^{-1}\| \leq C_\theta$ for all $\lambda \in \overline{\Str_\theta}^c.$
Similarly to the sectorial case, one defines $f(B)$ for $f \in \HI(\Str_\theta)$ satisfying a decay for $|\Re \lambda| \to \infty$ by a Cauchy integral formula, and says that $B$ has a bounded $\HI(\Str_\theta)$ calculus provided that $\|f(B)\| \leq C \|f\|_{\infty,\theta},$ in which case $f \mapsto f(B)$ extends to a bounded homomorphism $\HI(\Str_\theta) \to B(X).$
We refer to \cite{CDMY} and \cite[Chapter 4]{Haasa} for details.
We call $B$ $0$-strip-type if $B$ is $\omega$-strip-type for all $\omega > 0.$

There is an analogous statement to Lemma \ref{Lem Hol} which holds for a $0$-strip-type operator $B$ and $\Str_\omega$ in place of $A$ and $\Sigma_\omega,$ and $\Hol(\Str_\omega) = \{ f : \Str_\omega \to \C :\: \exists n \in \N :\: (\rho \circ \exp)^n f \in \HI(\Str_\omega) \},$
where $\rho(\lambda) = \lambda ( 1 + \lambda)^{-2},$ see \cite[p.~91-96]{Haasa}.

In fact, $0$-strip-type operators and $0$-sectorial operators with bounded $\HI(\Str_\omega)$ and bounded $\HI(\Sigma_\omega)$ calculus are in one-one correspondence by the following lemma.
For a proof we refer to \cite[Proposition 5.3.3., Theorem 4.3.1 and Theorem 4.2.4, Lemma 3.5.1]{Haasa}.

\begin{lem}
Let $B$ be a $0$-strip-type operator and assume that there exists a $0$-sectorial operator $A$ such that $B = \log(A)$.
This is the case if $B$ has a bounded $\HI(\Str_\omega)$ calculus for some $\omega < \pi.$
Then for any $f \in \bigcup_{0 < \omega < \pi} \Hol(\Str_\omega)$ one has
\[ f(B) = (f\circ \log)(A). \]
Note that the logarithm belongs to $\Hol(\Sigma_\omega)$ for any $\omega \in (0,\pi).$
Conversely, if $A$ is a $0$-sectorial operator that has a bounded $\HI(\Sigma_\omega)$ calculus for some $\omega \in (0,\pi),$ then $B = \log(A)$ is a $0$-strip-type operator.
\end{lem}

Let $B$ be a $0$-strip-type operator and $\alpha > \frac12.$
We say that $B$ has a (bounded) $\Sobolev^\alpha_2$ calculus if there exists a constant $C > 0$ such that
\[ \|f(B)\| \leq C \|f\|_{\Sobolev^\alpha_2} \quad (f\in \bigcap_{\omega > 0} \HI(\Str_\omega) \cap \Sobolev^\alpha_2) .\]
In this case, by density of $\bigcap_{\omega > 0} \HI(\Str_\omega) \cap \Sobolev^\alpha_2$ in $\Sobolev^\alpha_2$,
the definition of $f(B)$ can be continuously extended to $f \in \Sobolev^\alpha_2.$

Assume that $B$ has a $\Sobolev^\alpha_2$ calculus.
Let $f \in \Sobolev^\alpha_{2,\loc}.$
We define the operator $f(B)$ to be the closure of
\[ \begin{cases}
D_B \subset X & \longrightarrow X \\
x & \longmapsto \sum_{n \in \Z} (\equi_n f)(B)x,
\end{cases}
\]
where $D_B = \{ x \in X :\: \exists N \in \N:\: \equi_n(B)x= 0 \quad (|n| \geq N) \}$ and $(\equi_n)_{n \in \Z}$ is an equidistant partition of unity.

Then there holds an analogous version of Lemma \ref{Lem Soaloc calculus}.
Let $\Wa = \{ f \in L^2_{\text{loc}}(\R) : \: \|f\|_{\Wa} = \sup_{n \in \Z} \| \equi_n f \|_{\Soa} < \infty \}.$
Note that $\Wa$ is contained in $\Soaloc.$
Thus the $\Soaloc$ calculus for $B$ enables us to define the $\Wa$ calculus:
Let $\alpha > \frac12$ and $B$ be a $0$-strip-type operator.
We say that $B$ has an ($R$-bounded) $\Wa$ calculus if there exists a constant $C > 0$ such that
\[ \left\{ f(B) :\: f \in \bigcap_{\omega > 0} \HI(\Str_\omega) \cap \Wa,\:\|f\|_{\Wa} \leq 1 \right\}\text{ is }(R\text{-)bounded}. \]
The strip-type version of the main Theorems \ref{Thm Characterizations Hoermander calculus} and  \ref{Thm Sea Ha} reads as follows.

\begin{thm}\label{Thm Vertical chain of conditions strip}
Let $B$ be $0$-strip-type operator with $\HI$ calculus on some Banach space with property $(\alpha).$
Denote $U(t)$ the $C_0$-group generated by $iB$ and $R(\lambda,B)$ the resolvents of $B.$
For $\alpha > \frac12,$ consider the condition
\begin{equation}
B \text{ has an }R\text{-bounded }\Wor^\alpha_2\text{ calculus.}
\tag*{$(C_2)_\alpha$}
\end{equation}
Furthermore, we consider the conditions
\begin{enumerate}
\item[(a${)_\alpha}$] The family $(\tma U(t):\: t \in \R)$ is $R[L^2(\R)]$-bounded.
\item[(b${)_\alpha}$] The family $(R(t+ic,B):\: t \in \R)$ is $R[L^2(\R)]$-bounded for any $c \neq 0$ and its bound grows at most like $|c|^{-\alpha}$ for $c \to 0.$
\end{enumerate}
Then for all $\epsilon > 0,$
\[(C_2)_\alpha \Longleftrightarrow (\text{a})_\alpha \Longrightarrow (\text{b})_\alpha \Longrightarrow (C_2)_{\alpha + \epsilon}\]
\end{thm}

\begin{proof}
Consider the $0$-sectorial operator $A = e^B.$
Then $(C_2)_\alpha \Longleftrightarrow (\text{a})_\alpha$ follows from Theorems \ref{Thm Characterizations Hoermander calculus} and \ref{Thm Sea Ha}.

\noindent
$(\text{a})_\alpha \Longrightarrow (\text{b})_\alpha:$
Let $R_c = |c|^\alpha R[L^2]( R(t+ic,B):\:t\in \R).$
We have to show $\sup_{c\neq 0} R_c < \infty.$
Applying Lemma \ref{Lem Transformation R[E]-bounded} with $K$ the Fourier transform and its inverse, we get
\[ R_c = \begin{cases}
R[L^2](c^\alpha e^{ct} U(t):\:t < 0),&c > 0,\\
R[L^2](|c|^\alpha e^{ct}U(t):\: t > 0),& c < 0.
\end{cases}
\]
For $t < 0,\,\sup_{c > 0} c^\alpha e^{ct} = \sup_{c > 0} \tma \ta c^\alpha e^{-|ct|} \lesssim \tma.$
Thus,
$\sup_{c>0} R[L^2](c^\alpha e^{ct} U(t):\:t < 0) \lesssim R[L^2](\tma U(t):\: t < 0) < \infty.$
The part $c < 0$ is estimated similarly.

\noindent
$(\text{b})_\alpha \Longrightarrow (\text{a})_{\alpha + \epsilon}:$
Let $R_c$ be as before.
Split $\langle t \rangle^{-(\alpha + \epsilon)} U(t)$ into the parts $t \geq 1,t \leq -1,|t| < 1,$ and further $t \geq 1$ into $t \in [2^n,2^{n+1}],n \in \N_0.$
Then $\tma \lesssim 2^{-n\alpha} \lesssim 2^{-n\alpha} e^{-2^{-n}t},$ and by Lemma \ref{Lem R[E]-bounded technical} (2),
\begin{align}
R[L^2](\langle t \rangle^{-(\alpha + \epsilon)} U(t):\:t \geq 1)
& \leq \sum_{n=0}^\infty 2^{-n\epsilon}R[L^2]( 2^{-n\alpha} e^{-2^{-n}t} U(t):\:t\in [2^n,2^{n+1}])
\nonumber \\
& \leq \sum_{n=0}^\infty 2^{-n\epsilon} \sup_{c < 0}R_c < \infty.
\nonumber
\end{align}
The estimate for $t \leq -1$ can be handled similarly.
It remains to estimate $R[L^2](\langle s \rangle^{-(\alpha + \epsilon)} U(s) :\: |s| < 1) \cong R[L^2](U(s) :\: |s| < 1).$
We have assumed that $X$ has property $(\alpha).$
Then the fact that $B$ has an $\HI$ calculus implies that $\{U(s):\: |s| < 1\}$ is $R$-bounded \cite[Corollary 6.6]{KaW2}.
For $f \in L^2([-1,1]),$ we have $\|f\|_1 \leq C \|f\|_2,$ and consequently,
\[ \left\{ \int_{-1}^1 f(s)U(s) ds :\: \|f\|_2 \leq 1 \right\} \subset C \left\{ \int_{-1}^1 f(s) U(s) ds :\: \|f\|_1 \leq 1 \right\}.\]
In other words, $(U(s):\: |s| < 1)$ is $R[L^2]$-bounded.
\end{proof}

\section*{Acknowledgment}
We thank the anonymous referee for the careful reading of the manuscript.


\begin{thebibliography}{100}

\bibitem{Alex}
G.~Alexopoulos.
\newblock Spectral multipliers on Lie groups of polynomial growth.
\newblock {\em Proc. Am. Math. Soc.} 120(3):973--979, 1994.

\bibitem{BeL}
J.~Bergh and J.~L\"ofstr\"om.
\newblock {\em Interpolation spaces. An introduction.}
\newblock Grundlehren der mathematischen Wissenschaften, 223.
  Berlin etc.: Springer, 1976.

\bibitem{BoCl}
A.~Bonami and J.-L.~Clerc.
\newblock Sommes de Ces\`aro et multiplicateurs des d\'eveloppements en harmoniques sph\'eriques.
\newblock {\em Trans. Amer. Math. Soc.} 183:223--263, 1973.

\bibitem{Bou}
J.~Bourgain.
\newblock Vector valued singular integrals and the $H^1-$BMO duality.
\newblock {\em Probability theory and harmonic analysis (Cleveland, Ohio, 1983)
Monogr. Textbooks Pure Appl. Math., Vol. 98}, p.1--19 Dekker, New York, 1986.

\bibitem{CDMY}
M.~Cowling, I.~Doust, A.~McIntosh and A.~Yagi.
\newblock Banach space operators with a bounded $H\sp \infty$ functional
  calculus.
\newblock {\em J. Aust. Math. Soc., Ser. A} 60(1):51--89, 1996.

\bibitem{CPSW}
P.~Cl\'ement, B.~de Pagter, F.~Sukochev and H.~Witvliet.
\newblock Schauder decomposition and multiplier theorems.
\newblock {\em Studia Math.} 138:135--163, 2000.

\bibitem{COSY}
P.~Chen, E.M.~Ouhabaz, A.~Sikora and L.~Yan.
\newblock Restriction estimates, sharp spectral multipliers and endpoint estimates for Bochner-Riesz means
\newblock Preprint available on http://arxiv.org/abs/1202.4052

\bibitem{DiJT}
J.~Diestel, H.~Jarchow and A.~Tonge.
\newblock {\em Absolutely summing operators.}
\newblock Cambridge Studies in Advanced Mathematics, 43. Cambridge: Cambridge
  Univ. Press, 1995.

\bibitem{DW}
M.~Duelli and L.~Weis.
\newblock Spectral projections, Riesz transforms and $\HI$-calculus for bisectorial operators.
\newblock {\em Nonlinear elliptic and parabolic problems, Progress in Nonlinear Differential Equations Appl.} 64:99--111, Birkh\"auser, Basel, 2005.

\bibitem{Duon}
X.~T. Duong.
\newblock From the $L\sp 1$ norms of the complex heat kernels to a H\"ormander
  multiplier theorem for sub-Laplacians on nilpotent Lie groups.
\newblock {\em Pac. J. Math.} 173(2):413--424, 1996.

\bibitem{DuOS}
X.~T. Duong, E.~M. Ouhabaz and A.~Sikora.
\newblock Plancherel-type estimates and sharp spectral multipliers.
\newblock {\em J. Funct. Anal.} 196(2):443--485, 2002.

\bibitem{GaMi}
J.~E. Gal\'e and P.~J. Miana.
\newblock $H^{\infty}$ functional calculus and Mikhlin-type multiplier
  conditions.
\newblock {\em Can. J. Math.} 60(5):1010--1027, 2008.

\bibitem{HaKu}
B.~H. Haak and P.~C. Kunstmann.
\newblock Admissibility of unbounded operators and wellposedness of linear systems in Banach spaces.
\newblock {\em Int. Equ. Op. Theory} 55(4):497--533, 2006.

\bibitem{Haas}
M.~Haase.
\newblock Functional calculus for groups and applications to evolution
  equations.
\newblock {\em Journal Of Evolution Equations} 7(3):529--554, 2007.

\bibitem{Haasa}
M.~Haase.
\newblock {\em The functional calculus for sectorial operators.}
\newblock Operator Theory: Advances and Applications, 169. Basel: Birkh\"auser, 2006.

\bibitem{Haasb}
M.~Haase.
\newblock A transference principle for general groups and functional calculus on UMD spaces.
\newblock {\em Math. Ann.} 345:245--265, 2009.

\bibitem{Haasc}
M.~Haase.
\newblock Transference principles for semigroups and a theorem of Peller.
\newblock {\em J. Funct. Anal.} 261(10):2959--2998, 2011.

\bibitem{HyVe}
T.~Hyt\"onen and M.~Veraar.
\newblock $R$-boundedness of smooth operator-valued functions.
\newblock {\em Integral Equations Oper. Theory} 63(3):373--402, 2009.

\bibitem{Hoa}
L.~H\"ormander.
\newblock {\em The analysis of linear partial differential operators. I.
  Distribution theory and Fourier analysis. 2nd ed.}
\newblock Grundlehren der Mathematischen Wissenschaften, 256. Berlin etc.:
  Springer, 1990.

\bibitem{KaW2}
N.~Kalton and L.~Weis.
\newblock The $H^\infty$-calculus and square function estimates, preprint.

\bibitem{Kr}
C.~Kriegler.
\newblock Spectral multipliers, $R$-bounded homomorphisms, and analytic diffusion semigroups.
\newblock {PhD-thesis, online at
http://digbib.ubka.uni-karlsruhe.de/volltexte/1000015866}

\bibitem{Kr2}
C.~Kriegler.
\newblock Spectral multipliers for wave operators.
\newblock Submitted, Preprint available on http://arxiv.org/abs/1210.4261

\bibitem{Kr3}
C.~Kriegler.
\newblock H\"ormander Functional Calculus for Poisson Estimates.
\newblock {\em Int. Equ. Op. Theory} 80(3):379--413, 2014.

\bibitem{KrW1}
C.~Kriegler and L.~Weis.
\newblock Paley-Littlewood Decomposition for Sectorial Operators and Interpolation Spaces.
\newblock Submitted.

\bibitem{KrW2}
C.~Kriegler and L.~Weis.
\newblock Spectral multiplier theorems via $\HI$ calculus and norm bounds.
\newblock Preprint.

\bibitem{KrLM}
C.~Kriegler and C.~Le Merdy.
\newblock Tensor extension properties of $C(K)$-representations and applications to unconditionality.
\newblock {\em J. Aust. Math. Soc.} 88(2):205--230, 2010.

\bibitem{KuUh}
P.C.~Kunstmann and M.~Uhl.
\newblock Spectral multiplier theorems of H\"ormander type on Hardy and Lebesgue spaces.
\newblock Preprint available on http://arxiv.org/abs/1209.0358

\bibitem{KuWe}
P.~C. Kunstmann and L.~Weis.
\newblock Maximal $L_p$-regularity for parabolic equations, Fourier multiplier
  theorems and $H^\infty$-functional calculus.
\newblock {\em Functional analytic methods for
  evolution equations. Based on lectures given at the autumn school on
  evolution equations and semigroups, Levico Terme, Trento, Italy, October
  28--November 2, 2001.} Berlin: Springer, Lect. Notes Math. 1855,
  65--311, 2004.

\bibitem{Lebe}
N.~Lebedev.
\newblock {\em Special functions and their applications. Rev. engl. ed.
  Translated and edited by Richard A. Silverman.}
\newblock New York: Dover Publications, 1972.

\bibitem{LM}
C.~Le~Merdy.
\newblock On square functions associated to sectorial operators.
\newblock {\em Bull. Soc. Math. France} 132(1):137--156, 2004.

\bibitem{Ouha}
E.~M. Ouhabaz.
\newblock {\em Analysis of heat equations on domains.}
\newblock London Mathematical Society Monographs, 31. Princeton, NJ:
  Princeton University Press, 2005.

\bibitem{Rud}
W.~Rudin.
\newblock {\em Principles of mathematical analysis.}
\newblock Third edition. International Series in Pure and Applied Mathematics. McGraw-Hill Book Co., New York-Auckland-Düsseldorf, 1976.

\bibitem{RuSi}
T.~Runst and W.~Sickel.
\newblock {\em Sobolev spaces of fractional order, Nemytskij operators and
  nonlinear partial differential equations.}
\newblock de Gruyter Series in Nonlinear Analysis and Applications, 3. Berlin:
  de Gruyter, 1996.

\bibitem{Stem}
K.~Stempak.
\newblock Multipliers for eigenfunction expansions of some Schr\"odinger operators.
\newblock {\em Proc. Am. Math. Soc.} 93(3):477--482, 1985.

\bibitem{Triea}
H.~Triebel.
\newblock {\em Theory of function spaces.}
\newblock Monographs in Mathematics, 78. Basel etc.:
  Birkh\"auser, 1983.

\end{thebibliography}
\end{document}